\documentclass[11pt]{article}

\usepackage{mathtools}
\usepackage{color}
\usepackage{microtype}
\usepackage[usenames,dvipsnames,svgnames,table]{xcolor}
\usepackage{amsfonts,amsmath, amssymb,amsthm,amscd,mathrsfs}
\usepackage[height=9in,width=6.5in]{geometry}
\usepackage{verbatim}
\usepackage{tikz}
\usetikzlibrary{decorations.markings}
\usetikzlibrary{snakes}
\tikzset{->-/.style={decoration={
			markings,
			mark=at position .6 with {\arrow{>}}},postaction={decorate}}}
\usepackage{tkz-euclide}

\usepackage[margin=10pt,font=small,labelfont=bf]{caption}
\usepackage{latexsym}
\usepackage[pdfauthor={ },bookmarksnumbered,hyperfigures,colorlinks=true,citecolor=BrickRed,linkcolor=BrickRed,urlcolor=BrickRed,pdfstartview=FitH]{hyperref}

\usepackage{graphicx}
\usepackage{enumitem}

\usepackage[inline,nolabel]{showlabels}

\usepackage[capitalize]{cleveref}
\Crefname{fact}{Fact}{Facts}
\crefname{theorem}{Theorem}{Theorems}
\crefname{thm}{Theorem}{Theorems}
\crefname{lemma}{Lemma}{Lemmas}
\crefname{claim}{Claim}{Claims}
\crefname{lem}{Lemma}{Lemmas}
\crefname{remark}{Remark}{Remarks}
\crefname{prop}{Proposition}{Propositions}
\crefname{defn}{Definition}{Definitions}
\crefname{corollary}{Corollary}{Corollaries}
\crefname{conjecture}{Conjecture}{Conjectures}
\crefname{question}{Question}{Questions}
\crefname{chapter}{Chapter}{Chapters}
\crefname{section}{Section}{Sections}
\crefname{part}{Part}{Parts}
\crefname{figure}{Figure}{Figures}

\newtheorem{theorem}{Theorem}[section]
\newtheorem{corollary}[theorem]{Corollary}
\newtheorem{lemma}[theorem]{Lemma}
\newtheorem{proposition}[theorem]{Proposition}

\newtheorem{remark}[theorem]{Remark}

\numberwithin{equation}{section}

\newcommand{\bbR}{{\ensuremath{\mathbb R}} }

\newcommand{\cD}{{\ensuremath{\mathcal D}} }

\newcommand{\cP}{{\ensuremath{\mathcal P}} }

\DeclareMathSymbol{\leqslant}{\mathalpha}{AMSa}{"36} 
\DeclareMathSymbol{\geqslant}{\mathalpha}{AMSa}{"3E} 
\DeclareMathSymbol{\eset}{\mathalpha}{AMSb}{"3F}     
\renewcommand{\ge}{\;\geqslant\;}                   
\renewcommand{\leq}{\;\leqslant\;}                   
\renewcommand{\geq}{\;\geqslant\;}                   

\newcommand{\be}{\begin{equation}}
	\newcommand{\ee}{\end{equation}}
\newcommand\ba{\begin{align}}
	\newcommand\ea{\end{align}}

\usepackage{array,booktabs,longtable,enumitem}
\usepackage{xcolor}
\usepackage{soul} 

\newtheorem*{conjecture}{Conjecture}

\newcommand{\PP}{\mathbf P}
\newcommand{\EE}{\mathbf E}
\newcommand{\one}{\mathbf1}

\newcommand{\phit}{\varphi^{\mathrm T}}
\newcommand{\Zhat}{\widehat Z}
\newcommand{\Lhat}{\widehat\Lambda}
\DeclareMathOperator{\Cov}{Cov}
\DeclareMathOperator{\dist}{dist}
\DeclareMathOperator{\diam}{diam}
\DeclareMathOperator{\sgn}{sgn}
\DeclareMathOperator{\wt}{wt}

\newcommand{\Nhood}{\mathcal N}

\begin {document}
\author{
	Pengfei Tang\thanks{Center for Applied Mathematics and KL-AAGDM, Tianjin University, Tianjin, 300072, China.
		Email: \textsf{pengfei\_tang@tju.edu.cn}. Supported by the National Natural Science Foundation of China No. 12571151.}
	\qquad
	Zibo Zhang\thanks{School of Mathematics, Tianjin University, Tianjin, 300350, China.
		Email: \textsf{19932788533@163.com}.}
}
\date{\today}
\title{Negative correlation for the random-cluster model\newline below one on high-degree regular graphs}
\maketitle


\begin{abstract}

	The random-cluster model with cluster parameter \(0<q<1\) is conjectured
to exhibit negative dependence, but even pairwise negative correlation
between distinct edges remains open on general finite graphs. We first establish a reduction showing that the pairwise conjecture is
equivalent to its restriction to regular graph sequences of diverging
degrees $(d_n)$, even for edge pairs at distance
\(o(d_n/\log d_n)\).
	
	We then prove strict pairwise negative correlation for a broad class of
	high-degree regular graph sequences satisfying a two-scale
	edge-isoperimetric condition. More precisely, for every fixed
	\(p\in(0,1)\), all sufficiently large graphs in the sequence have
	strictly negative covariance between any two distinct edge indicators
	at distance \(o(d_n/\log d_n)\), uniformly in \(q\in[0,1)\) and in the
	choice of the two edges. In particular, the result includes the
	\(q=0\) endpoint, corresponding to Bernoulli bond percolation
	conditioned to be connected.
	
	The isoperimetric hypothesis is satisfied by  high-degree
	expander families and, with probability tending to one, by uniformly
	random regular graphs of diverging degree; its two-scale form also
	allows product geometries such as hypercubes and fixed-side
	high-dimensional discrete tori. The proof is based on a polymer
	representation and cluster expansion, together with a geometric
	identification of the leading contribution to the two-edge
	correlation.
\end{abstract}


\enlargethispage{3pt}
\tableofcontents


\section{Introduction}\label{sec:introduction}

	The random-cluster model, introduced by Fortuin and Kasteleyn, is a probability measure on subgraphs that includes independent bond percolation as the case \(q=1\) and is closely related to the ferromagnetic Potts model when \(q\) is a positive integer \cite{FortuinKasteleyn1972,Grimmett2006}. One of its fundamental features is the sharp contrast in correlation behavior across \(q=1\). For \(q\geq 1\), positive association provides a powerful and robust comparison principle. For \(0<q<1\), by contrast, the corresponding negative-dependence theory remains substantially less understood: even the sign of the covariance between two edges on a general finite graph remains unresolved \cite[Section~1]{Pemantle2000}; see also \cite[Section~4]{GrimmettWinkler2004}.

Let $G=(V,E)$ be a finite connected simple graph. A configuration
$\omega\subseteq E$ specifies the open edges, and $k(\omega)$ is the number
of connected components of $(V,\omega)$, including isolated vertices. For
$p\in(0,1)$ and $q>0$, the random-cluster measure  $\phi_{G,p,q}$ is given by
\begin{equation}\label{eq:usual-rc}
\phi_{G,p,q}(\omega)=\frac{1}{Z_{G,p,q}}
p^{|\omega|}(1-p)^{|E|-|\omega|}q^{k(\omega)},
\qquad \omega\subseteq E.
\end{equation}
Here $Z_{G,p,q}$ is the normalizing constant. The case $q=1$ is Bernoulli
bond percolation; integer $q\geq2$ gives the random-cluster representation
of the ferromagnetic $q$-states Potts model
\cite{FortuinKasteleyn1972}.

Our result includes the endpoint $q=0$. It is therefore convenient to use
\begin{align}
\Zhat_{G,p,q}
&:=\sum_{\eta\subseteq E}p^{|\eta|}(1-p)^{|E|-|\eta|}q^{k(\eta)-1},
\label{eq:extended-Z}\\
\mu_{G,p,q}(\omega)
&:=\frac{p^{|\omega|}(1-p)^{|E|-|\omega|}q^{k(\omega)-1}}
{\Zhat_{G,p,q}},\qquad q\ge 0,
\label{eq:extended-measure}
\end{align}
with the convention $0^0=1$. The fully open configuration shows that the denominator is
positive. For $q>0$, $Z_{G,p,q}=q\Zhat_{G,p,q}$, so
$\mu_{G,p,q}=\phi_{G,p,q}$. At $q=0$ only connected configurations retain
positive weight, and hence
\begin{equation}\label{eq:connected-conditioned}
\mu_{G,p,0}=\PP_p\bigl(\,\cdot\mid(V,\omega)\text{ is connected}\bigr),
\end{equation}
where $\PP_p$ denotes Bernoulli$(p)$ bond percolation. In particular,
$\mu_{G,1/2,0}$ is uniform on the connected spanning subgraphs. This fixed-$p$
endpoint is different from the forest and spanning-tree limits, which
involve joint scalings of $p$ and $q$
\cite[Section~4]{GrimmettWinkler2004}. The finite-sum formula also gives
$\mu_{G,p,q}\to\mu_{G,p,0}$ as $q\downarrow0$.

\subsection{Negative dependence below one}\label{subsec:negative-dependence}
For $g\in E$, let $X_g=\one_{\{g\in\omega\}}$, and let $V(g)$ be its set
of endpoints. We use the edge distance
\begin{equation}\label{eq:edge-distance}
\dist_G(e,f):=\min_{x\in V(e),\,y\in V(f)}\dist_G(x,y).
\end{equation}
Thus distinct edges sharing a vertex have distance zero.

A probability measure $\nu$ on $\{0,1\}^{E}$ is \emph{pairwise negatively
	correlated} (p-NC, or NC) if
\begin{equation}\label{eq:negative-correlation}
\Cov_\nu(X_e,X_f)
=\nu(X_e=X_f=1)-\nu(X_e=1)\nu(X_f=1)\leq0
\qquad(e\neq f).
\end{equation}
It is \emph{negatively associated} (NA) if
$\Cov_\nu(F,H)\leq0$ for every pair of increasing real-valued functions
$F,H:\{0,1\}^E\to \bbR$ depending on disjoint edge sets. It is \emph{conditionally
	negatively associated} (CNA) if this remains true after conditioning on
any coordinate assignment of positive probability
\cite[Definition~2.7]{BorceaBrandenLiggett2009}.
An external field $\boldsymbol t\in(0,\infty)^E$ changes the measure to
\begin{equation}\label{eq:external-field}
\nu^{\boldsymbol t}(\omega)=
\frac{\nu(\omega)\prod_{g\in E}t_g^{X_g(\omega)}}
{\sum_{\eta\subseteq E}\nu(\eta)\prod_{g\in E}t_g^{X_g(\eta)}}.
\end{equation}
NA$^+$ and NC$^+$ require NA and NC, respectively, under every positive
external field; NC$^+$ is also called the \emph{Rayleigh property} \cite[Section~1]{KahnNeiman2010}.
CNA$^+$ requires NA to survive external fields, projections, and coordinate
conditioning. On a finite Boolean space, NA$^+$ and CNA$^+$ coincide:
projections preserve NA, and conditioning follows by taking suitable field
parameters to zero or infinity. The defining inequalities pass to these
limits; see \cite[Section~1]{KahnNeiman2010}. In particular,
\[
\mathrm{CNA}^+\equiv\mathrm{NA}^+
\ \Longrightarrow\ \mathrm{CNA}
\ \Longrightarrow\ \mathrm{NA}
\ \Longrightarrow\ \mathrm{p\text{-}NC}.
\]
The present work concerns the last property in this hierarchy: for the random-cluster model, even this basic two-edge covariance question is nontrivial and, on general finite graphs, has remained unresolved.

For $q\geq1$, the FKG lattice condition implies positive association
\cite{FortuinKasteleynGinibre1971,Grimmett2006}; the precise random-cluster
statement is \cite[Theorem~3.8]{Grimmett2006}. For $0<q<1$, conditioning
on all edges except $g=\{x,y\}$ instead gives
\begin{equation}\label{eq:single-edge-conditional}
\phi_{G,p,q}(X_g=1\mid(X_b)_{b\neq g})=
\begin{cases}
p,&x\text{ and }y\text{ are connected without }g,\\[1mm]
\displaystyle\frac{p}{p+q(1-p)},&\text{otherwise}.
\end{cases}
\end{equation}
The second probability is larger than $p$. Thus $q<1$ penalizes additional
components and favors edges that merge components. Different edges capable
of effecting the same merger may compete. This observation motivates
negative dependence, but does not determine unconditional covariances,
since connectivity is a global constraint. Kahn and Neiman conjectured
NA$^+$ for random-cluster measures below one
\cite[Conjecture~11]{KahnNeiman2010}. Even the following weaker statement
is open in general.

\begin{conjecture}[Pairwise negative correlation]\label{conj:pnc}
	For every finite graph $G$, every $p\in(0,1)$, and every $q\in(0,1)$,
	the random-cluster measure $\phi_{G,p,q}$ satisfies
	\eqref{eq:negative-correlation}.
\end{conjecture}

For comparison with generating-polynomial methods, set
\[
\mathcal Z_G(q,\boldsymbol x)
:=\sum_{\omega\subseteq E}q^{k(\omega)-1}\prod_{g\in\omega}x_g,
\qquad \boldsymbol x\in(0,\infty)^E.
\]
Direct differentiation gives, for $e\neq f$,
\begin{equation}\label{eq:covariance-derivative}
\Cov_{\mu_{G,p,q}}(X_e,X_f)
=\left.x_ex_f\partial_{x_e}\partial_{x_f}
\log\mathcal Z_G(q,\boldsymbol x)\right|_{x_g=p/(1-p),\ g\in E}.
\end{equation}
  Thus, the conjecture studied here asks for the relevant mixed logarithmic derivatives to be nonpositive along the homogeneous activity vector. By contrast, requiring the same inequality throughout the positive orthant \((0,\infty)^E\) corresponds to the NC$^+$  (Rayleigh) property.   At $q=1$ the polynomial is
	$\prod_{g\in E}(1+x_g)$,  so all mixed logarithmic derivatives vanish.

\subsection{Historical background and current status}\label{subsec:history}
The problem is closely related to negative dependence for random forests
and uniform connected  subgraphs. These models arise from the random-cluster model in suitable \(q\downarrow0\) regimes \cite[Section~1.5]{Grimmett2006}. Kahn proposed the uniform-forest
conjecture \cite[Conjecture~10.11]{Kahn2000}, and Pemantle identified the
random-cluster model below one as a basic test case for a general theory of
negative dependence \cite[Section~1]{Pemantle2000}. Grimmett and Winkler
formulated the forest and connected-subgraph conjectures in
\cite[Conjecture~1.1]{GrimmettWinkler2004} and discussed their random-cluster
interpretation in Section~4. Their term \emph{edge-negative association}
refers to p-NC, rather than to full NA. They also verified the forest
inequality computationally for all graphs with at most eight vertices and
for graphs with nine vertices and at most eighteen edges
\cite[Theorem~1.3]{GrimmettWinkler2004}.

Spanning trees admit stronger conclusions. The balanced-matroid argument
of Feder and Mihail gives NA for the uniform spanning-tree measure
\cite{FederMihail1992}. A complementary description is supplied by the
transfer-impedance formula of Burton and Pemantle: its finite-graph form
is \cite[Theorem~1.1]{BurtonPemantle1993}, whereas their Theorem~4.2 treats
periodic graphs. The resulting determinantal structure is special and does
not give an analogous general argument for forests, connected subgraphs,
or random-cluster measures with $0<q<1$.

A second approach uses generating polynomials and graph or matroid
operations. Semple and Welsh proved that independence correlation and
spanning correlation are preserved under series--parallel extensions
\cite[Proposition~3.7]{SempleWelsh2008}. Wagner's Potts--Rayleigh framework
establishes NC$^+$ for series--parallel graphs through
\cite[Example~5.1 and Theorem~5.8(d)]{Wagner2008}; this result is also
recorded immediately after Conjecture~11 in \cite{KahnNeiman2010}.
These closure arguments give exact inequalities, but apply to restricted
classes of graphs.

Stable and Lorentzian polynomials provide different structural criteria.
A measure is \emph{strongly Rayleigh} when its multivariate generating
polynomial is real stable, that is, nonzero whenever every variable has
positive imaginary part. Borcea, Br\"and\'en, and Liggett proved that
strongly Rayleigh measures are CNA$^+$
\cite[Theorem~4.9]{BorceaBrandenLiggett2009}. This includes spanning-tree
measures, but does not cover general random-cluster measures on graphs
with cycles; see their Section~3.4. Br\"and\'en and Huh proved that the
homogenized multivariate Tutte polynomial of a matroid is Lorentzian for
$0<q\leq1$ \cite[Theorem~4.10]{BrandenHuh2020}. Their Propositions~4.21
and~4.25 give the $2$-Rayleigh property for uniform independent-set
measures. Applied to a uniform forest $F$, this yields
\[
\PP(e,f\in F)\leq2\PP(e\in F)\PP(f\in F).
\]
The desired inequality has constant $1$, so Lorentzianity alone does not
settle the sign question.

There are also results for particular geometries and parameter regimes.
Stark proved p-NC for uniform forests on sufficiently large complete graphs
\cite{Stark2011}. \c{C}i\c{c}eksiz's master's thesis reports a complete-graph
result on negative edge dependence for the random-cluster measure
\cite[Abstract]{Ciceksiz2021}. For the arboreal gas, Huang proved negative
correlation of adjacent edges on each fixed finite connected graph when
the edge weight is sufficiently large \cite{Huang2024}.
In \cite[Theorems~1.4--1.6]{TangZhang2026}, we proved p-NC on sufficiently
large complete graphs for uniform connected spanning subgraphs, forests
with a fixed number of components, and connected spanning subgraphs with
fixed excess. On wired trees, Park proved CNA$^+$ for the vector of
parallel branch-connectivity indicators
\cite[Proposition~6.1]{Park2026}, and NA for increasing observables on
branch-separated edge sets, including the corresponding infinite-volume
statement \cite[Theorem~6.4 and Corollary~6.5]{Park2026}. The separation
restriction in the latter result does not give full NA for arbitrary
disjoint edge sets. These advances exploit additional algebraic or
geometric structure; the pairwise question for general finite graphs
remains unresolved.

\subsection{Why a geometric hypothesis is needed}\label{subsec:motivation}
A large degree might appear to simplify the problem by itself. The
following reduction, proved in Section~\ref{sec:reduction}, shows why a
geometric hypothesis is needed for the approach developed here.

\begin{proposition}\label{thm:reduction}
	Fix $p\in(0,1)$ and $q\in[0,1)$. The following statements are equivalent.
	\begin{enumerate}
		\item[(i)] For every finite connected simple graph $H$ and every pair of distinct
		edges $e,f\in E(H)$, $\Cov_{\mu_{H,p,q}}(X_e,X_f)\leq0$.
		\item[(ii)] For every sequence $(G_n)$ of finite connected simple $d_n$-regular
		graphs with $d_n\to\infty$, for all sufficiently large $n$ and every pair
		of distinct edges $e,f\in E(G_n)$,
		$\Cov_{\mu_{G_n,p,q}}(X_e,X_f)\leq0$.
		\item[(iii)] For every sequence $(G_n)$ as in \textup{(ii)} and every sequence of
		nonnegative integers $(r_n)$ satisfying $r_n\log d_n=o(d_n)$, for all
		sufficiently large $n$ and all distinct $e,f\in E(G_n)$ with
		$\dist_{G_n}(e,f)\leq r_n$, one has
		$\Cov_{\mu_{G_n,p,q}}(X_e,X_f)\leq0$.
	\end{enumerate}
\end{proposition}

Thus a theorem for unrestricted high-degree regular graphs, even at the
indicated distance scale, would imply the full finite-graph conjecture.
The construction proving Proposition~\ref{thm:reduction} preserves the original
edge law exactly, but introduces bottlenecks. Our isoperimetric assumption
excludes these bottlenecks and supplies the quantitative information needed
to control a cluster expansion.

\subsection{The graph condition and the main theorem}\label{subsec:main}
Let $G_n=(V_n,E_n)$ be finite connected simple $d_n$-regular graphs, and
write $N_n=|V_n|$. Assume that
\begin{equation}\label{eq:degree-limit}
d_n\longrightarrow\infty.
\end{equation}
For $S\subseteq V_n$, write $E_{G_n}(S):=\{\{x,y\}\in E_n\colon x,y \in S  \}$ for its internal edges,
$e_{G_n}(S)=|E_{G_n}(S)|$, and
\[
\partial_{G_n}S:=\{\{x,y\}\in E_n:x\in S,\ y\in V_n\setminus S\}.
\]
Let $(r_n)$ be nonnegative integers such that
\begin{equation}\label{eq:distance-assumption}
r_n\log d_n=o(d_n).
\end{equation}
For a fixed $h>0$, set
\begin{equation}\label{eq:L-definition}
L_n:=\left\lceil\frac{16}{h}(r_n+2)\right\rceil,
\qquad m_{L_n}(S):=\min\{|S|,L_n\}.
\end{equation}
Our hypothesis is the two-scale edge-isoperimetric inequality
\begin{equation}\label{eq:two-scale-isoperimetry}
|\partial_{G_n}S|\geq h\bigl(|S|+d_n m_{L_n}(S)\bigr)
\qquad\text{for all }1\leq|S|\leq N_n/2,
\end{equation}
for every sufficiently large $n$.

\begin{theorem}\label{thm:main}
	Fix $p\in(0,1)$. Let $(G_n)$ satisfy \eqref{eq:degree-limit} and
	\eqref{eq:two-scale-isoperimetry}, and let $(r_n)$ satisfy
	\eqref{eq:distance-assumption}, with $L_n$ defined by
	\eqref{eq:L-definition}. For all sufficiently large $n$, every
	$q\in[0,1)$, and all distinct $e,f\in E_n$ with
	$\dist_{G_n}(e,f)\leq r_n$,
	\begin{equation}\label{eq:main-conclusion}
	\Cov_{\mu_{G_n,p,q}}(X_e,X_f)<0.
	\end{equation}
	The threshold in $n$ is uniform in $q,e,f$.
\end{theorem}

The inequality is strict, and one threshold works throughout the half-open
interval $q\in[0,1)$. The result includes the connected-conditioned measure
at $q=0$, including its uniform version when $p=1/2$. The permitted distance
may diverge: any $r_n=o(d_n/\log d_n)$ is allowed when the corresponding
isoperimetric hypothesis holds. Uniformity of the threshold does not mean
that the covariance is bounded away from zero uniformly in $q$. For a
fixed graph it tends to zero as $q\uparrow1$.

More precisely, the proof identifies a shortest-path contribution to the
logarithmic correlation.  Writing \(\ell=\dist_G(e,f)\), let \(\nu_\ell(e,f)\) denote the number of shortest paths directed from \(V(e)\) to \(V(f)\) when \(\ell\geq1\). When \(\ell=0\), i.e.,  \(e\) and \(f\) have a common endpoint, we adopt the convention \(\nu_0(e,f)=1\).   Put $s=1-p$ and $\kappa_q=p(p+qs)/s^2$. Then
\begin{equation}\label{eq:intro-asymptotic}
\log\frac{\mu_{G,p,q}(X_e=X_f=1)}
{\mu_{G,p,q}(X_e=1)\mu_{G,p,q}(X_f=1)}
=-(1-q)\nu_\ell(e,f)s^{d(\ell+1)}\kappa_q^\ell\bigl(1+o(1)\bigr),
\end{equation}
where $o(1)$ is uniform in $q\in[0,1)$ and in the admissible edge
pair, for fixed $p,h$.
This is established by the explicit remainder estimate in
Proposition~\ref{prop:asymptotic} and the comparison in
Section~\ref{subsec:main-proof}.

\subsection{Motivation for the two-scale condition and examples}\label{subsec:examples-intro}

	The hypothesis is designed to include two expansion regimes. For $|S|\leq L_n$, it
	requires degree-scale expansion,
	\[
	|\partial S|\geq h(d_n+1)|S|,
	\]
	whereas for $|S|\geq L_n$ it requires only
	\[
	|\partial S|\geq h|S|+hd_nL_n.
	\]
	The distinction is motivated by product graphs. Indeed, consider the $n$-dimensional hypercube $\mathsf Q_n=\{0,1\}^n$ and the half-cube obtained by fixing one coordinate,
	\[
	S=\{x=(x_1,\ldots,x_n)\in\{0,1\}^n, x_1=0\}.
	\]
	Then $|S|=2^{n-1}$, and each vertex of $S$ has exactly one neighbor outside $S$, obtained by changing its first coordinate from $0$ to $1$. Hence $|\partial S|=|S|$. Since $\mathsf Q_n$ has degree $d_n=n$,
	\[
	\frac{|\partial S|}{d_n|S|}=\frac1n\longrightarrow0.
	\]
	Thus degree-scale expansion cannot hold uniformly over all scales.

If $|\partial S|\geq h_0d_n|S|$ for all $|S|\leq N_n/2$ and a fixed
$h_0>0$, then \eqref{eq:two-scale-isoperimetry} follows with $h=h_0/2$.
This includes complete graphs, balanced complete multipartite graphs,
fixed-dimensional Hamming graphs with growing alphabet, Paley graphs,
and independent-set blow-ups of a fixed nontrivial connected regular graph.
Section~\ref{subsec:spectral-examples} gives definitions and verifies these
claims using the one-sided adjacency spectral gap. All these particular
families have bounded diameter, so the result applies to every distinct
edge pair for sufficiently large degree.

As the preceding half-cube example shows, the hypercube and fixed-side high-dimensional tori are not uniformly expanding in this normalized sense. The point of the two-scale condition is precisely that it still accommodates these product geometries. We record the relevant bounds here and give the common Cartesian-product verification in Section~\ref{subsec:product-examples}.
For the $n$-dimensional hypercube $\mathsf Q_n$, the classical inequality
\cite{Harper1964,Hart1976} is
\[
|\partial_{\mathsf Q_n}S|\geq|S|\bigl(n-\log_2|S|\bigr).
\]
It yields the required condition when $L_n\log n=o(n)$, and hence covers
distances $r_n=o(n/\log n)$. The product-graph inequality of Diskin and
Samotij \cite[Section~3.5, equation~(8)]{DiskinSamotij2025} gives a common
proof for hypercubes and fixed-side tori. In fact, the same verification
applies to Cartesian powers of any fixed nontrivial connected regular
graph; see Corollary~\ref{cor:products}.

For the random-regular example, fix integer pairs $(N_n,d_n)$ with
$0\leq d_n\leq N_n-1$, $N_nd_n$ even, and $d_n\to\infty$.
Choose $G_n$ uniformly from the simple $d_n$-regular graphs on the fixed
labeled vertex set $[N_n]\coloneq \{1,2,\ldots,N_n\}$. Proposition~\ref{prop:random-iso} shows that, with probability at least $1-N_n^{-1}$, these graphs have normalized edge expansion bounded away from zero. Corollary~\ref{cor:expanders} therefore yields the strict covariance inequality, with the same probability, simultaneously for all \(q\in[0,1)\) and all edge pairs at distance at most \(r_n\), whenever \(r_n\log d_n=o(d_n)\).

\subsection{Cluster expansion and the proof strategy}\label{subsec:strategy}
We use a cluster expansion to identify the leading contribution to the dependence between two marked edges. The expansion is based on a polymer representation in which the polymers are connected vertex sets. Since the polymer weights may have either sign, we use Ueltschi's complex-measure formulation of the cluster expansion \cite[Theorem~1]{Ueltschi2004}, together with its rooted estimate \cite[equation~(4)]{Ueltschi2004}. The required Koteck\'y--Preiss conditions are verified for each activity used below.

To obtain the polymer representation, we first decompose each constrained partition function into two terms. One term has an exact polymer representation, while the other is supported on configurations in which every open component has size at most \(N_n/2\). We show that the latter contribution, as well as the contribution of polymers larger than \(N_n/2\), is exponentially small in \(N_n\) and \(d_nL_n\).

The main estimate needed to control the polymer expansion is a bound on the number of connected vertex sets with a given edge boundary. Theorem~3 of Carlson, Davies, Fraiman, Kolla, Potukuchi, and Yap
\cite{CarlsonEtAl2024} gives a count of the form $d_n^{O_h(b/d_n)}$ for sets
containing a fixed vertex and having $b$ boundary edges. For fixed $p\in(0,1)$ and
large $d_n$, this growth is dominated by the factor $(1-p)^b$ contributed by the closed boundary edges. Combined with \eqref{eq:two-scale-isoperimetry}, this yields the activity estimates needed for the Koteck\'y--Preiss criterion and for controlling large clusters. We also introduce suitable exponential weights to obtain tail bounds in the total polymer size, together with a size-marked expansion used in the shortest-path calculation.


We then compare the four cluster expansions corresponding to the two marked edges. All terms whose support fails to meet both endpoint sets cancel (Lemma~\ref{lem:cancellation}). If the two edges are at distance \(\ell\), every surviving cluster has total polymer size at least \(\ell+1\). At the minimum size, its support is an induced shortest path between the two endpoint sets, and the polymers partition this path into intervals (Lemma~\ref{lem:minimal-support}). The total contribution of these minimal clusters is computed exactly in Proposition~\ref{prop:path-leading} and gives the negative leading term in \eqref{eq:intro-asymptotic}. All remaining clusters are bounded in absolute value, and their total contribution is of smaller order. The leading term and the error retain the common factor \(1-q\), which makes the comparison uniform over \(q\in[0,1)\).

Thus the leading interaction between the two marked edges is determined by shortest paths joining their endpoint sets, while the isoperimetric condition makes the remaining contributions negligible in comparison. This yields the asymptotic formula \eqref{eq:intro-asymptotic} and, in particular, strict negative correlation for the edge pairs covered by Theorem~\ref{thm:main}.

%

\subsection{Notation and organization}\label{subsec:conventions}
Except in the reduction argument and the examples, we fix $n$ and suppress
it from the notation:
\[
\begin{gathered}
G=(V,E),\quad N=|V|,\quad d=d_n,\quad L=L_n,\\
E(S)=E_{G_n}(S),\quad e(S)=|E(S)|,\quad \partial S=\partial_{G_n}S.
\end{gathered}
\]
We write $s=1-p$, $a=-\log s>0$, and $\tau=1-q$. All logarithms without
a specified base are natural. Constants $c,C,C_0,\ldots$ are positive and
finite. Unless otherwise stated, they may depend on $p,h$, but not on
$n,q,e,f$, the permitted distance, or the particular graph. A constant is
fixed within each statement and proof; its value may change between
statements.

The assumptions imply
\begin{equation}\label{eq:scale-consequences}
(r_n+1)\log d_n=o(d_n),\qquad
L_n=O_h(r_n+1),\qquad L_n\log d_n=o(d_n).
\end{equation}
Applying the graph condition to a singleton gives $h<1$. Since $G_n$ is
simple, $N_n\geq d_n+1$, and therefore, for all sufficiently large $n$,
\begin{equation}\label{eq:L-range}
r_n+1<L_n<d_n/2<N_n/2.
\end{equation}
A notation index, including the constrained partition functions, modified
activities, and spectral quantities, is given in
Appendix~\ref{sec:notation-index}, beginning on page~\pageref{sec:notation-index}.

The paper is organized as follows. Section~\ref{sec:reduction} proves
Proposition~\ref{thm:reduction}. Section~\ref{sec:polymers} develops the exact
partition-function identities. Section~\ref{sec:estimates} establishes the
activity and cluster estimates, including all applications of Ueltschi's
theorem. Section~\ref{sec:paths} proves the cancellation and shortest-path
formulas. Section~\ref{sec:proof-examples} completes the proof of
Theorem~\ref{thm:main} and verifies the graph condition in the examples.


\section{Reduction to high-degree regular graphs}\label{sec:reduction}
This section proves Proposition~\ref{thm:reduction}. The construction attaches
regularizing graphs at cut vertices, so that the original edge law and all
original vertex distances are preserved exactly.

\begin{lemma}[Factorization at a cut vertex]\label{lem:one-sum}
	Let $G_i=(V_i,E_i)$, $i=1,2$, be finite connected simple graphs with
	$V_1\cap V_2=\{v\}$ and $E_1\cap E_2=\varnothing$. Write
	$G_1\vee_vG_2=(V_1\cup V_2,E_1\cup E_2)$. For $p\in(0,1)$ and
	$q\in[0,1)$,
	\begin{align}
	\Zhat_{G_1\vee_vG_2,p,q}
	&=\Zhat_{G_1,p,q}\Zhat_{G_2,p,q},\label{eq:one-sum-Z}\\
	\mu_{G_1\vee_vG_2,p,q}
	&=\mu_{G_1,p,q}\otimes\mu_{G_2,p,q}.
	\label{eq:one-sum-law}
	\end{align}
\end{lemma}
\begin{proof}
	Every configuration has a unique decomposition
	$\omega=\omega_1\cup\omega_2$, with $\omega_i\subseteq E_i$. Exactly the
	two open components containing $v$ are merged in the one-vertex sum.
	Consequently,
	\[
	k_{G_1\vee_vG_2}(\omega)-1
	=\bigl(k_{G_1}(\omega_1)-1\bigr)
	+\bigl(k_{G_2}(\omega_2)-1\bigr).
	\]
	The edge counts also add. Thus the unnormalized weight in
	\eqref{eq:extended-measure} is the product of the two corresponding
	weights. This remains valid at $q=0$, since both exponents are nonnegative
	integers and $0^{b+c}=0^b0^c$ under our convention. Summing proves
	\eqref{eq:one-sum-Z}; normalization proves \eqref{eq:one-sum-law}.
\end{proof}

\begin{lemma}[A degree-completion gadget]\label{lem:gadget}
	Let $D\geq3$ be odd and $1\leq r\leq D$. There is a finite connected
	simple graph $J_{D,r}$ with a specified vertex $\rho$ such that
	\[
	\deg_{J_{D,r}}(\rho)=r,
	\qquad \deg_{J_{D,r}}(x)=D\quad(x\neq\rho).
	\]
	It has $D+2$ vertices if $r$ is even, and $D+3$ vertices if $r$ is odd.
\end{lemma}
\begin{proof}
	We construct a graph $F$ whose complement is $J_{D,r}$.
	
If $r$ is even, take $W=A\mathbin{\dot\cup}B$, the disjoint union of
		$A$ and $B$, with $|A|=D+1-r$ and $|B|=r$, and set
		$V(J_{D,r})=V(F)=W\cup\{\rho\}$. In $F$, join $\rho$ to
	every vertex of $A$, put a perfect matching on $B$, and add no other
	edges. Both required cardinalities are even. Every vertex in $W$ has
	$F$-degree one, and $\deg_F(\rho)=D+1-r$. Since the total number of
	vertices is $D+2$, complementation gives the required degrees.
	
 If $r$ is odd, take $W=A\mathbin{\dot\cup}B$ with $|A|=D+2-r$ and $|B|=r$, and set
		$V(J_{D,r})=V(F)=W\cup\{\rho\}$. Now $|A|$ is even and at least two. Choose
	$a_1,a_2\in A$, enumerate $B=\{b_1,\ldots,b_r\}$, and put the path
	\[
	a_1,b_1,\ldots,b_r,a_2
	\]
	on $W$, together with a perfect matching on $A\setminus\{a_1,a_2\}$.
	The matching may be empty. Adjoin $\rho$ and join it to all of $A$.
	Every vertex of $W$ now has $F$-degree two, whereas
	$\deg_F(\rho)=D+2-r$. There are $D+3$ vertices, so the complement again
	has root degree $r$ and all other degrees $D$.
	
	In either construction the root has positive degree. Every component of
	$J_{D,r}$ therefore contains a nonroot vertex and hence at least $D+1$
	vertices. Two components would require at least $2(D+1)>D+3$ vertices.
	Thus $J_{D,r}$ is connected.
\end{proof}

\begin{proof}[Proof of Proposition~\ref{thm:reduction}]
	The implications \textup{(i)}$\Rightarrow$\textup{(ii)}$\Rightarrow$
	\textup{(iii)} are immediate. Assume \textup{(iii)}, and fix a finite
	connected simple graph $H$ and distinct edges $e,f\in E(H)$.
	For an odd integer $D>\max\{3,\Delta(H)\}$, where $\Delta(H)$ is the
	maximum degree of $H$, set
	\[
	\kappa_v:=D-\deg_H(v)\qquad(v\in V(H)).
	\]
	At each vertex $v$, attach a disjoint copy of $J_{D,\kappa_v}$ by
	identifying its root with $v$. Denote the resulting graph by
	\begin{equation}\label{eq:regularization}
	R_D(H):=H\mathop{\vee}_{v\in V(H)}J_{D,\kappa_v}.
	\end{equation}
	Each original vertex has degree $\deg_H(v)+\kappa_v=D$; every other
	vertex has degree $D$ by Lemma~\ref{lem:gadget}. The graph is finite,
	connected, and simple, and $H$ remains an induced subgraph.
	
	Repeated use of Lemma~\ref{lem:one-sum} gives the product decomposition
	\begin{equation}\label{eq:regularization-law}
	\mu_{R_D(H),p,q}
	=\mu_{H,p,q}\otimes\bigotimes_{v\in V(H)}\mu_{J_{D,\kappa_v},p,q}.
	\end{equation}
	In particular, the marginal law of all original edges is unchanged, and
	\begin{equation}\label{eq:regularization-cov}
	\Cov_{\mu_{R_D(H),p,q}}(X_e,X_f)
	=\Cov_{\mu_{H,p,q}}(X_e,X_f).
	\end{equation}
	Likewise, a simple path with endpoints in $H$ cannot enter an attached
	graph and return: it would have to revisit its root. Hence every shortest
	path between original vertices lies in $H$, giving
	\begin{equation}\label{eq:regularization-distance}
	\dist_{R_D(H)}(e,f)=\dist_H(e,f).
	\end{equation}
	
	Choose $D_n=2n+2\max\{3,\Delta(H)\}+1$, let $G_n=R_{D_n}(H)$, and set
	$r_n=\dist_H(e,f)$. Then $D_n\to\infty$, $r_n\log D_n=o(D_n)$, and
	\eqref{eq:regularization-distance} makes the pair admissible in
	\textup{(iii)}. Applying that statement and then
	\eqref{eq:regularization-cov} proves \textup{(i)}.
\end{proof}

The construction also explains its incompatibility with our geometric
hypothesis. Let $S_{D,v}$ be the nonroot vertices of one attached gadget.
Then
\begin{equation}\label{eq:gadget-cut}
D+1\leq|S_{D,v}|\leq D+2,
\qquad |\partial_{R_D(H)}S_{D,v}|=\kappa_v\leq D.
\end{equation}
Since $H$ has at least three vertices,
$|V(R_D(H))|\geq3(D+2)$, so $|S_{D,v}|<|V(R_D(H))|/2$.
For any fixed $r_0\geq0$ and $h>0$, put
$L_0=\lceil16(r_0+2)/h\rceil$. For sufficiently large $D$,
$|S_{D,v}|>L_0$, and
\[
h\bigl(|S_{D,v}|+D\min\{|S_{D,v}|,L_0\}\bigr)
\geq hDL_0\geq16D(r_0+2)>|\partial S_{D,v}|.
\]
Thus the regularized graphs fail \eqref{eq:two-scale-isoperimetry}, even
at a fixed distance scale. Furthermore, an original edge and an edge in
an attached gadget are independent by \eqref{eq:regularization-law}.
Strict negative correlation therefore cannot be inferred from high degree
alone.


\section{Partition functions and polymers}\label{sec:polymers}
The identities in this section are finite and exact. They require no
asymptotic assumption. We use the fixed-graph notation of
Section~\ref{subsec:conventions}.

\subsection{Constrained partition functions}\label{subsec:constrained}
For $D\subseteq E$ with $|D|\leq2$, let $\EE_D$ denote expectation under
$\PP_p$ conditioned on $D\subseteq\omega$. Define
\begin{equation}\label{eq:constrained-Z}
\Zhat_D(q):=\EE_D[q^{k(\omega)-1}]
=\sum_{\eta\subseteq E\setminus D}
p^{|\eta|}s^{|E\setminus D|-|\eta|}q^{k(\eta\cup D)-1}.
\end{equation}
The fully open configuration gives $\Zhat_D(q)>0$ for $q\in[0,1)$.
For $D=\varnothing$ this is \eqref{eq:extended-Z}. Factoring out the
probability of the forced edges yields
\begin{equation}\label{eq:edge-probabilities}
\mu_{G,p,q}(X_e=1)=p\frac{\Zhat_{\{e\}}}{\Zhat_\varnothing},
\qquad
\mu_{G,p,q}(X_e=X_f=1)=p^2\frac{\Zhat_{\{e,f\}}}{\Zhat_\varnothing}.
\end{equation}
We suppress the argument $q$ when there is no ambiguity. Hence
\begin{equation}\label{eq:covariance-Z}
\Cov_{\mu_{G,p,q}}(X_e,X_f)
=\frac{p^2}{\Zhat_\varnothing^2}
\bigl(\Zhat_{\{e,f\}}\Zhat_\varnothing
-\Zhat_{\{e\}}\Zhat_{\{f\}}\bigr).
\end{equation}
Set
\begin{equation}\label{eq:log-correlation}
\Lhat_q(e,f):=
\log\Zhat_{\{e,f\}}+\log\Zhat_\varnothing
-\log\Zhat_{\{e\}}-\log\Zhat_{\{f\}}.
\end{equation}
All logarithms here are real. Equations~\eqref{eq:edge-probabilities} and
\eqref{eq:covariance-Z} imply
\begin{align}
\Lhat_q(e,f)
&=\log\frac{\mu_{G,p,q}(X_e=X_f=1)}
{\mu_{G,p,q}(X_e=1)\mu_{G,p,q}(X_f=1)},
\label{eq:log-correlation-prob}\\
\sgn\Cov_{\mu_{G,p,q}}(X_e,X_f)&=\sgn\Lhat_q(e,f).
\label{eq:sign-equivalence}
\end{align}
It therefore suffices to determine the sign of the logarithmic difference.

\subsection{Separating configurations without a giant component}\label{subsec:giant}
Let $\mathcal K_-(\omega)$ be the family of vertex sets of the open
components with at most $N/2$ vertices, and put
\[
k_-(\omega):=|\mathcal K_-(\omega)|,
\qquad
\cD:=\{\omega:\text{every open component has size at most }N/2\}.
\]
On $\cD^c$ there is exactly one component larger than $N/2$, so
$k=k_-+1$. On $\cD$ one has $k=k_-\geq2$.
With $\tau=1-q$, define
\begin{equation}\label{eq:TQ}
T_q(\omega):=
\begin{cases}
q^{k_-(\omega)-1},&\omega\in\cD,\\
0,&\omega\notin\cD.
\end{cases}
\end{equation}
The exponent in the first case is positive. Thus this definition is valid
also at $q=0$, and pointwise
\begin{equation}\label{eq:component-decomposition}
q^{k(\omega)-1}=q^{k_-(\omega)}+\tau T_q(\omega).
\end{equation}
Writing
\begin{equation}\label{eq:YR}
Y_D:=\EE_D[q^{k_-(\omega)}],\qquad R_D:=\EE_D[T_q(\omega)],
\end{equation}
we obtain the exact decomposition
\begin{equation}\label{eq:Z-Y-R}
\Zhat_D=Y_D+\tau R_D.
\end{equation}
Moreover, $Y_D\geq p^{|E|-|D|}>0$, because the fully open configuration
has $k_-=0$. The remainder is nonnegative and supported on $\cD$.

\subsection{The signed polymer representation}\label{subsec:polymer-identity}
Let $\cP$ be the set of all nonempty $S\subseteq V$ such that $G[S]$ is
connected; its elements are called \emph{polymers}. Two polymers $S,T\in\cP$ are \emph{compatible}, written
$S\sim T$, if they are disjoint and there is no edge of $G$ between them.
Otherwise write $S\not\sim T$. In particular, a polymer is incompatible
with itself.

For $H\subseteq E(S)$, let $c(S,H)$ be the number of components of
$(S,H)$, and define
\[
\chi_-(S,H):=
\one_{\{\text{every component of }(S,H)\text{ has at most }N/2
	\text{ vertices}\}}.
\]
Set $c(\varnothing,\varnothing)=0$ and
$\chi_-(\varnothing,\varnothing)=1$. 
For $D_S:=D\cap E(S)$, define the \emph{activity} of a polymer $S\in\cP$ by
\begin{equation}\label{eq:activity}
w_D(S):=
\begin{cases}
\displaystyle s^{|\partial S|}
\sum_{\substack{H\subseteq E(S)\\D_S\subseteq H}}
\chi_-(S,H)(q-1)^{c(S,H)}p^{|H|-|D_S|}s^{e(S)-|H|},
&D\cap\partial S=\varnothing,\\[3mm]
0,&D\cap\partial S\neq\varnothing.
\end{cases}
\end{equation}
 These activities need not be positive, and we occasionally write $w=w_\varnothing$. The cutoff applies to the components of $(S,H)$, not to $S$ itself. In particular, polymers with $|S|>N/2$ may still contribute if all components of $(S,H)$ have size at most $N/2$, and hence cannot be discarded.

\begin{lemma}[Exact polymer identity]\label{lem:polymer-identity}
	For every $D\subseteq E$ with $|D|\leq2$ and every $q\in[0,1)$,
	\begin{equation}\label{eq:polymer-Y}
	Y_D=\sum_{\substack{\mathscr S\subseteq\cP\\
			\mathscr S\ \mathrm{pairwise\ compatible}}}
	\prod_{S\in\mathscr S}w_D(S).
	\end{equation}
	The empty family contributes one.
\end{lemma}
\begin{proof}
	Expand once for every small open component:
	\begin{equation}\label{eq:select-components}
	q^{k_-(\omega)}
	=\prod_{K\in\mathcal K_-(\omega)}(1+(q-1))
	=\sum_{\mathscr A\subseteq\mathcal K_-(\omega)}(q-1)^{|\mathscr A|}.
	\end{equation}
	For a selected subfamily, let $U$ be the union of its component vertex
	sets. Every edge of $\partial U$ is closed. A contribution is therefore
	impossible if $D\cap\partial U\neq\varnothing$. Otherwise, if
	$H=\omega\cap E(U)$, the selected components are exactly the components
	of $(U,H)$, all of size at most $N/2$. Summation over edges entirely
	outside $U$ gives one under the conditioned Bernoulli law. Consequently,
	\begin{equation}\label{eq:bD-sum}
	Y_D=\sum_{U\subseteq V}b_D(U),
	\end{equation}
	where $b_D(U)$ is the expression on the right of \eqref{eq:activity}
	with $S$ replaced by $U$, without requiring $G[U]$ to be connected.
	The conventions above give $b_D(\varnothing)=1$.
	
	Let $S_1,\ldots,S_j$ be the vertex sets of the connected components of
	$G[U]$. They are pairwise compatible. There are no graph edges between
	them, so boundaries, internal edge counts, and component counts add, while
	$\chi_-$ factors. The internal edge sums in $b_D(U)$ therefore separate:
	\begin{equation}\label{eq:bD-factorization}
	b_D(U)=\prod_{i=1}^j w_D(S_i).
	\end{equation}
	If a forced edge crosses $\partial U$, both sides are zero; otherwise the
	factorization follows term by term. Finally, the map sending $U$ to the
	components of $G[U]$ is a bijection from subsets of $V$ to pairwise
	compatible polymer families, with inverse given by union. Substituting
	\eqref{eq:bD-factorization} into \eqref{eq:bD-sum} proves the identity.
\end{proof}


\section{Activity estimates and cluster expansions}\label{sec:estimates}
From now on the graph satisfies \eqref{eq:two-scale-isoperimetry}, and
$n$ is sufficiently large for \eqref{eq:L-range} to hold. All estimates
are uniform in $q\in[0,1)$ and in the forcing set $D$ with $|D|\leq2$.
The factor $\tau=1-q$ will be kept explicit.

\subsection{Connected sets with a prescribed boundary}\label{subsec:cut-count}
For a graph $J$, a vertex $v$, and an integer $b\geq0$, let
\[
N_J(v,b):=\#\{S\subseteq V(J):v\in S,\ 1\leq|S|\leq|V(J)|/2,
\ J[S]\text{ connected},\ |\partial_JS|=b\}.
\]
We use the following connected-cut count.

\begin{lemma}[Carlson--Davies--Fraiman--Kolla--Potukuchi--Yap]
	\label{lem:cut-count}
	There is an absolute constant $C_{\rm cut}>0$ such that, if $J$ is a
	finite $d$-regular graph with $d\geq2$ and
	$|\partial_JA|\geq\eta|A|$ for all $1\leq|A|\leq|V(J)|/2$, then
	\begin{equation}\label{eq:cut-count}
	N_J(v,b)\leq d^{C_{\rm cut}(1+1/\eta)b/d}.
	\end{equation}
\end{lemma}
\begin{proof}
	This is \cite[Theorem~3]{CarlsonEtAl2024}. The boundary denoted by
	$\nabla(A)$ there is $\partial_JA$ here, and connectedness means that
	$J[A]$ is connected. The expansion parameter in that theorem is
	unnormalized: its hypothesis is $|\partial_JA|\geq\eta|A|$, not
	$\eta d|A|$.
\end{proof}

Our assumption implies both $|\partial S|\geq h|S|$ and
$|\partial S|\geq h(d+1)$ for nonempty $S$ of size at most $N/2$.
Thus, with $C_h=C_{\rm cut}(1+1/h)$,
\begin{equation}\label{eq:cut-count-exponential}
N_G(v,b)\leq\exp\left(C_h\frac{\log d}{d}\,b\right).
\end{equation}

\begin{lemma}[Boundary summability]\label{lem:boundary-sum}
	For every fixed $\xi>0$, there is $c_\xi=c_\xi(\xi,h)>0$ such that
	\begin{equation}\label{eq:boundary-sum}
	\sup_{v\in V}\sum_{\substack{S\ni v,\ 1\leq|S|\leq N/2\\G[S]\text{ connected}}}
	e^{-\xi|\partial S|}\leq e^{-c_\xi d}
	\end{equation}
	for all sufficiently large $n$.
\end{lemma}
\begin{proof}
	Put $b_0=\lceil h(d+1)\rceil$. Grouping sets by their boundary size and
	using \eqref{eq:cut-count-exponential}, the sum is at most
	\[
	\sum_{b\geq b_0}
	\exp\left[-\left(\xi-C_h\frac{\log d}{d}\right)b\right]
	\leq\frac{e^{-\xi h(d+1)/2}}{1-e^{-\xi/2}}
	\]
	for large $d$. The prefactor is independent of $d$ and $v$, so the last
	expression is at most $e^{-\xi hd/4}$ for all sufficiently large $d$.
\end{proof}

\subsection{Small and large polymer activities}\label{subsec:activity-estimates}
\begin{lemma}[A crude activity bound]\label{lem:crude-activity}
	For every $S\in\cP$,
	\begin{equation}\label{eq:crude-activity}
	|w_D(S)|\leq\tau s^{|\partial S|}.
	\end{equation}
\end{lemma}
\begin{proof}
If $D\cap\partial S\neq\varnothing$, then $w_D(S)=0$ by definition. Otherwise
	$c(S,H)\geq1$, so $|q-1|^{c(S,H)}\leq\tau$. Drop $\chi_-$ and sum the
	remaining conditional Bernoulli weights over $H\supseteq D_S$. Their sum
	is $(p+s)^{|E(S)\setminus D_S|}=1$.
\end{proof}

We also record an elementary count used for small clusters.
\begin{lemma}\label{lem:connected-set-count}
	If $J$ has maximum degree at most $d\geq1$, then
	\begin{equation}\label{eq:connected-set-count}
	\#\{S\subseteq V(J):v\in S,\ |S|=k,\ J[S]\text{ connected}\}
	\leq(4d)^{k-1}\qquad(k\geq1).
	\end{equation}
\end{lemma}
\begin{proof}
	Fix neighbor orders, and use depth-first search to assign a rooted ordered
	spanning tree to each set. There are at most $4^{k-1}$ rooted ordered tree
	shapes on $k$ vertices, by the Catalan bound
	\cite[Chapter~1]{Stanley2015}. With the root fixed at $v$, each successive
	vertex has at most $d$ possible images. Allowing noninjective images only
	increases the count, proving the claim.
\end{proof}

Regularity and simplicity give, for $k=|S|$,
\begin{equation}\label{eq:boundary-degree}
|\partial S|=d|S|-2e(S)\geq k(d-k+1).
\end{equation}
For the remainder of the proof put
\begin{equation}\label{eq:lambda-theta}
\lambda:=\frac{ah}{8}>0,\qquad \theta:=s^{1/2}=e^{-a/2},
\qquad \cP_{\leq}:=\{S\in\cP:|S|\leq N/2\}.
\end{equation}

\begin{lemma}[Small polymers]\label{lem:small-activity}
	There is $c>0$ such that
	\begin{equation}\label{eq:small-weighted}
	\sup_{v\in V}\sum_{\substack{S\in\cP_{\leq}\\S\ni v}}
	|w_D(S)|e^{\lambda|S|+\lambda d m_L(S)}\leq\tau e^{-cd}
	\end{equation}
	for all sufficiently large $n$.
\end{lemma}
\begin{proof}
	The graph condition gives
	$\lambda(|S|+dm_L(S))\leq a|\partial S|/8$.
	By Lemma~\ref{lem:crude-activity}, the summand is therefore at most
	$\tau e^{-7a|\partial S|/8}$. Applying Lemma~\ref{lem:boundary-sum} with
	$\xi=7a/8$ finishes the proof.
\end{proof}

\begin{lemma}[Large polymers]\label{lem:large-activity}
	For all sufficiently large $n$,
	\begin{equation}\label{eq:large-weighted}
	\sum_{\substack{S\in\cP\\|S|>N/2}}|w_D(S)|e^{\lambda|S|}
	\leq\tau e^{-\lambda N/4-\lambda dL}.
	\end{equation}
\end{lemma}
\begin{proof}
	Fix $S$ with $|S|>N/2$. In \eqref{eq:activity}, group internal
	configurations according to their open-component vertex sets
	$K_1,\ldots,K_j$. The cutoff requires $K_i\in\cP_{\leq}$, these sets are
	disjoint, and they cover $S$. Let $b_{\rm int}$ count graph edges joining
	different blocks. The edges of $\partial S$ and these interblock edges
	are closed. Restoring the Bernoulli weights of the at most two forced
	edges costs at most $p^{-2}$. Dropping the internal connectivity and
	forcing requirements then gives
	\[
	|w_D(S)|\leq\tau p^{-2}
	\sum_{\substack{\{K_1,\ldots,K_j\}\text{ disjoint}\\
			K_i\in\cP_{\leq},\ \bigcup_iK_i=S}}
	s^{|\partial S|+b_{\rm int}}.
	\]
	Since $\sum_i|\partial K_i|=|\partial S|+2b_{\rm int}$ and $0<s<1$,
	\begin{equation}\label{eq:large-block-bound}
	|w_D(S)|\leq\tau p^{-2}
	\sum_{\substack{\{K_1,\ldots,K_j\}\text{ disjoint}\\
			K_i\in\cP_{\leq},\ \bigcup_iK_i=S}}
	\prod_i\theta^{|\partial K_i|}.
	\end{equation}
	The block families are unordered.
	
	For $K\in\cP_{\leq}$, define
	\begin{equation}\label{eq:rho}
	\rho(K):=e^{2\lambda|K|+2\lambda d m_L(K)}\theta^{|\partial K|}.
	\end{equation}
	The graph condition implies $\rho(K)\leq e^{-a|\partial K|/4}$.
	Lemma~\ref{lem:boundary-sum} and double counting therefore give
	\begin{equation}\label{eq:rho-sums}
	\sup_{v\in V}\sum_{K\ni v}\rho(K)\leq e^{-c_1d},
	\qquad \sum_{K\in\cP_{\leq}}\rho(K)\leq Ne^{-c_1d}.
	\end{equation}
	A block family covering $S$ satisfies
	$\sum_i|K_i|=|S|$ and $\sum_i m_L(K_i)\geq L$: either one block has
	size at least $L$, or the latter sum equals $|S|>N/2>L$.
	It follows that
	\[
	e^{\lambda|S|}\prod_i\theta^{|\partial K_i|}=e^{-\lambda |S|-2\lambda d\sum_i m_L(K_i) }\prod_i\rho(K_i)
	\leq e^{-\lambda N/2-2\lambda dL}\prod_i\rho(K_i).
	\]
	A block family determines its union $S$ uniquely. We may therefore sum
	\eqref{eq:large-block-bound} over $S$, remove the disjointness and union
	constraints, and expand a finite product to obtain
	\begin{align*}
	\sum_{|S|>N/2}|w_D(S)|e^{\lambda|S|}
	&\leq\tau p^{-2}e^{-\lambda N/2-2\lambda dL}
	\prod_{K\in\cP_{\leq}}(1+\rho(K))\\
	&\leq\tau p^{-2}
	\exp\{-\lambda N/2-2\lambda dL+Ne^{-c_1d}\}.
	\end{align*}
	For sufficiently large $d$, $e^{-c_1d}\leq\lambda/4$ and
	$2\log(p^{-1})\leq\lambda dL$. These inequalities prove
	\eqref{eq:large-weighted}.
\end{proof}

\subsection{Ueltschi's theorem  in the present notation}\label{subsec:ueltschi}

	For $S,T\in\cP$, let
	\begin{equation}\label{eq:def-pair-fcn}
	\zeta(S,T):=-\one_{\{S\not\sim T\}}.
	\end{equation}
	Then $\zeta$ is symmetric and $|1+\zeta(S,T)|\leq1$. For an ordered
	polymer tuple $\Gamma=(S_1,\ldots,S_j)\in\cP^j$, repetitions being
	allowed, let $I_\Gamma$ be the incompatibility graph on the index set
	$[j]$, with $\{i,k\}\in E(I_\Gamma)$ if and only if $S_i\not\sim S_k$.
	Define the Ursell coefficient
	\begin{equation}\label{eq:ursell}
	\phit(\Gamma):=
	\sum_{\substack{H\text{ connected}\\V(H)=[j]}}
	\prod_{\{i,k\}\in E(H)}\zeta(S_i,S_k),
	\end{equation}
	with $ \phit(S_1)=1$. Only connected spanning subgraphs of $I_\Gamma$ can contribute to the
	sum, and hence $ \phit(\Gamma)=0$ whenever $I_\Gamma$ is disconnected. Our normalization has no factorial in the
	definition of $ \phit$: Ueltschi's coefficient $\varphi$ in his equation
	(2) in \cite{Ueltschi2004} is $ \phit/j!$.

\begin{proposition}[Finite hard-core specialization of Ueltschi's theorem]
	\label{prop:ueltschi}
	Let $\mathcal A\subseteq\cP$, let $v:\mathcal A\to\mathbb C$, and suppose
	there is $b:\mathcal A\to[0,\infty)$ such that
	\begin{equation}\label{eq:KP}
	\sum_{\substack{T\in\mathcal A\\T\not\sim S}}
	|v(T)|e^{b(T)}\leq b(S)\qquad(S\in\mathcal A).
	\end{equation}
	Then the partition function
	\[
	Y_{\mathcal A}(v):=
	\sum_{\substack{\mathscr S\subseteq\mathcal A\\
			\mathscr S\ \mathrm{pairwise\ compatible}}}
	\prod_{S\in\mathscr S}v(S)
	\]
	satisfies $Y_{\mathcal A}(v)=\exp B_{\mathcal A}(v)$, where
	\begin{equation}\label{eq:abstract-expansion}
	B_{\mathcal A}(v)=
	\sum_{j\geq1}\frac1{j!}\sum_{(S_1,\ldots,S_j)\in\mathcal A^j}
	\phit(S_1,\ldots,S_j)\prod_{i=1}^jv(S_i)
	\end{equation}
	converges absolutely. Moreover, for every fixed $S_0\in\mathcal A$,
	\begin{equation}\label{eq:rooted-bound}
	\sum_{j\geq1}\frac1{(j-1)!}
	\sum_{(S_2,\ldots,S_j)\in\mathcal A^{j-1}}
	|\phit(S_0,S_2,\ldots,S_j)|\prod_{i=2}^j|v(S_i)|
	\leq e^{b(S_0)}.
	\end{equation}
\end{proposition}
	\begin{proof}
		Apply \cite[Theorem~1, equations~(3)--(4)]{Ueltschi2004} to the finite
		discrete space $\mathcal A$, equipped with the complex measure
		\[
		\mathfrak m_v:=\sum_{S\in\mathcal A}v(S)\delta_S,
		\]
		and with pair function $\zeta$ given by \eqref{eq:def-pair-fcn}. Its total
		variation is
		\[
		|\mathfrak m_v|=\sum_{S\in\mathcal A}|v(S)|\delta_S.
		\]
		Thus Ueltschi's condition~(3) is precisely the Koteck\'y--Preiss condition \eqref{eq:KP} \cite{KoteckyPreiss1986}, while the
		additional integrability condition is automatic since $\mathcal A$ is
		finite.
		
		We first identify Ueltschi's partition function with
		$Y_{\mathcal A}(v)$. In the present discrete setting, his equation~(1) in \cite{Ueltschi2004}
		becomes
		\[
		1+\sum_{j\geq1}\frac1{j!}
		\sum_{(S_1,\ldots,S_j)\in\mathcal A^j}
		\prod_{i=1}^j v(S_i)
		\prod_{1\leq i<k\leq j}\bigl(1+\zeta(S_i,S_k)\bigr).
		\]
		Since
		\[
		1+\zeta(S,T)=\one_{\{S\sim T\}},
		\]
		the pair-factor product is nonzero exactly when
		$S_1,\ldots,S_j$ are pairwise compatible. In particular, repeated
		polymers give zero because $S\not\sim S$. Hence every nonzero term
		consists of $j$ distinct, pairwise compatible polymers. Each unordered
		compatible family $\mathcal S\subseteq\mathcal A$ of cardinality $j$
		occurs in exactly $j!$ orderings, which cancels the prefactor $1/j!$.
		Consequently, Ueltschi's partition function is
		\[
		\sum_{\substack{\mathcal S\subseteq\mathcal A\\
				\mathcal S\ {\rm pairwise\ compatible}}}
		\prod_{S\in\mathcal S}v(S)
		=Y_{\mathcal A}(v).
		\]
		
		Ueltschi's Theorem~1 in \cite{Ueltschi2004} therefore applies to $Y_{\mathcal A}(v)$. His
		coefficient $\varphi$ in equation~(2) is normalized by $1/j!$ relative to
		our Ursell coefficient:
		\[
		\varphi(S_1,\ldots,S_j)
		=\frac1{j!}\phit(S_1,\ldots,S_j).
		\]
		Substitution into his cluster expansion gives \eqref{eq:abstract-expansion},
		and his rooted estimate~(4) gives \eqref{eq:rooted-bound}. This proves
		the proposition.
	\end{proof}

\subsection{Verification for all activities used below}\label{subsec:KP-verification}
For $S\subseteq V$, write
\begin{equation}\label{eq:neighborhood}
\Nhood(S):=S\cup\{x\in V\setminus S:\{x,y\}\in E
\text{ for some }y\in S\}.
\end{equation}
The definition of incompatibility implies that whenever  two polymers $T$ and $S$ satisfy $T\not\sim S$, we have $T\cap\Nhood(S)\neq\varnothing$. We further observe that
$|\Nhood(S)|\leq(d+1)|S|$.

\begin{lemma}[The original and forced-edge activities]\label{lem:KP-base}
	For all sufficiently large $n$, the activity $w_D$ satisfies \eqref{eq:KP} on $\cP$
	with $b(S)=\lambda|S|$:
		\begin{equation}\label{eq:KP-wD}
		\sum_{\substack{T\in\cP\\T\not\sim S}}
		|w_D(T)|e^{\lambda|T|}
		\leq\lambda|S|,
		\qquad S\in\cP.
		\end{equation}
	Moreover, for some $c>0$,
	\begin{equation}\label{eq:base-root-sum}
	\sup_{v\in V}\sum_{\substack{S\in\cP\\S\ni v}}|w_D(S)|e^{\lambda|S|}
	\leq\tau e^{-cd}.
	\end{equation}
	This includes $w=w_\varnothing$.
\end{lemma}
\begin{proof}
	For $|S|\leq N/2$, drop the additional nonnegative exponential weight
	from \eqref{eq:small-weighted}. For larger polymers use
	\eqref{eq:large-weighted} and $N\geq d+1$. After decreasing $c$, their
	sum gives \eqref{eq:base-root-sum}. Consequently,
	\begin{equation}\label{eq:KP-neighborhood-bound}
	\sum_{\substack{T\in\cP\\T\not\sim S}}|w_D(T)|e^{\lambda|T|}
	\leq\sum_{v\in\Nhood(S)}\sum_{\substack{T\in\cP\\T\ni v}}|w_D(T)|e^{\lambda|T|}
	\leq(d+1)|S|\tau e^{-cd}.
	\end{equation}
	Since $\tau\leq1$ and $(d+1)e^{-cd}\to0$, the last quantity is at most
	$\lambda|S|$ for large $d$.
\end{proof}

Proposition~\ref{prop:ueltschi} now gives
\begin{equation}\label{eq:log-Y-expansion}
\log Y_D=\sum_{j\geq1}\frac1{j!}
\sum_{(S_1,\ldots,S_j)\in\cP^j}
\phit(S_1,\ldots,S_j)\prod_iw_D(S_i).
\end{equation}
The series is real and absolutely convergent. Its exponential is $Y_D>0$,
so it is the real logarithm, with no branch choice left unresolved.

\begin{lemma}[Exponentially weighted small-polymer activities]
	\label{lem:KP-tilted}
	Recall $\cP_{\leq}$ from \eqref{eq:lambda-theta}. 
	For $S\in\cP_{\leq}$, let
	\begin{equation}\label{eq:tilted-activity}
	\widetilde w_D(S):=e^{\lambda d m_L(S)}w_D(S).
	\end{equation}
	For all sufficiently large $n$, these activities satisfy \eqref{eq:KP}
	on $\cP_{\leq}$ with $b(S)=\lambda|S|$: 
		\begin{equation}\label{eq:KP-tilde-wD}
		\sum_{\substack{T\in\cP_{\leq}\\T\not\sim S}}
		|\widetilde w_D(T)|e^{\lambda|T|}
		\leq\lambda|S|,
		\qquad S\in\cP_{\leq}\,.
		\end{equation}
	Moreover, for some $c>0$,
	\begin{equation}\label{eq:tilted-root-sum}
	\sup_{v\in V}\sum_{\substack{S\in\cP_{\leq}\\S\ni v}}
	|\widetilde w_D(S)|e^{\lambda|S|}\leq\tau e^{-cd}.
	\end{equation}
\end{lemma}
\begin{proof}
	Equation~\eqref{eq:tilted-root-sum} is \eqref{eq:small-weighted}.
	The neighborhood argument in \eqref{eq:KP-neighborhood-bound}, now with
	$T\in\cP_{\leq}$, proves the criterion \eqref{eq:KP-tilde-wD}.
\end{proof}

\begin{lemma}[Size marking and restriction of the polymer family]
	\label{lem:KP-marker}
	
		Let $(\mathcal X,v)$ be either $(\cP,w_D)$ or
		$(\cP_{\leq},\widetilde w_D)$, and let
		$\mathcal A\subseteq\mathcal X$. For $u\in\mathbb C$, define
		\begin{equation}\label{eq:size-marked-activity}
		v_u(S):=u^{|S|}v(S),\qquad S\in\mathcal A.
		\end{equation}
		Uniformly over these choices, for all sufficiently large $n$, the
		following hold.
		
		\smallskip
		\noindent
		{\rm (i)} If $|u|\leq1$, then
		\begin{equation}\label{eq:KP-size-marked-unit}
		\sum_{\substack{T\in\mathcal A\\T\not\sim S}}
		|v_u(T)|e^{\lambda|T|}
		\leq\lambda|S|,
		\qquad S\in\mathcal A.
		\end{equation}
		
		\smallskip
		\noindent
		{\rm (ii)} If $|u|\leq R_*:=e^{\lambda/2}$, then
		\begin{equation}\label{eq:KP-size-marked-disk}
		\sum_{\substack{T\in\mathcal A\\T\not\sim S}}
		|v_u(T)|e^{\lambda|T|/2}
		\leq\frac{\lambda}{2}|S|,
		\qquad S\in\mathcal A.
		\end{equation}
		
		In particular, $Y_{\mathcal A}(v_u)$ has no zeros for $|u|\leq R_*$.
		The branch of its logarithm that vanishes at $u=0$ is given by
		\begin{equation}\label{eq:size-marked-cluster}
		\log Y_{\mathcal A}(v_u)
		=
		\sum_{j\geq1}\frac1{j!}
		\sum_{(S_1,\ldots,S_j)\in\mathcal A^j}
		\phit(S_1,\ldots,S_j)
		u^{\sum_i|S_i|}
		\prod_{i=1}^j v(S_i).
		\end{equation}
		The series converges absolutely and uniformly on $|u|\leq R_*$ and may
		be differentiated termwise on every smaller disk.	
\end{lemma}

\begin{proof}
	Restriction removes nonnegative summands from every absolute activity
	bound. Both choices of $v$ satisfy
	\[
	\sup_{x\in V}\sum_{\substack{T\in\mathcal A\\T\ni x}}
	|v(T)|e^{\lambda|T|}\leq\tau e^{-cd}
	\]
	by Lemmas~\ref{lem:KP-base} and~\ref{lem:KP-tilted}, with a common $c>0$.
	Part~\textup{(i)} follows by $|u|^{|T|}\leq1$. For part~\textup{(ii)},
	$|u|\leq e^{\lambda/2}$ implies
	\[
	|u|^{|T|}|v(T)|e^{\lambda|T|/2}
	\leq|v(T)|e^{\lambda|T|}.
	\]
	Consequently,
	\begin{equation}\label{eq:marked-KP-check}
	\sum_{\substack{T\in\mathcal A\\T\not\sim S}}
	|v_u(T)|e^{\lambda|T|/2}
	\leq(d+1)|S|\tau e^{-cd}
	\leq\frac{\lambda}{2}|S|
	\end{equation}
	for all sufficiently large $d$. This is precisely \eqref{eq:KP} for the
	modified activity and the stated control function. Symmetry,
	$|1+\zeta|\leq1$, and finite integrability are unchanged.
	
	Proposition~\ref{prop:ueltschi} applies. The series at the nonnegative
	majorant $R_*^{|S|}|v(S)|$ is an absolutely summable majorant for all
	$|u|\leq R_*$. Uniform convergence and analyticity on the open disk
	follow. Its exponential is $Y_{\mathcal A}(v_u)$, and its value at zero
	is zero since every polymer is nonempty. This identifies the logarithm
	and permits termwise differentiation on smaller disks.
\end{proof}

\begin{remark}\label{rem:formal-marker}
	Equation \eqref{eq:size-marked-cluster} also holds coefficientwise as a
		formal power-series identity. Here and below, $[u^m]F(u)$ denotes the
		coefficient of $u^m$ in the formal power series $F(u)$; this
		coefficient-extraction notation is distinct from the index-set notation
		$[j]=\{1,\ldots,j\}$. For fixed $m$, only tuples satisfying
		$\sum_i|S_i|=m$ contribute to $[u^m]$, and there are only finitely many
		such tuples. Hence coefficient extraction requires no convergence
		argument. The identity may also be obtained directly from the finite
		hard-core partition function by expanding the pair factors and applying
		the exponential formula to the connected components.
		Lemma~\ref{lem:KP-marker} is needed only for the analytic use of the
		expansion, where it provides absolute and locally uniform convergence
		and justifies termwise differentiation.

\end{remark}

\begin{table}[tbp]
	\centering\small
	\caption{Activities used in Ueltschi's theorem. Restrictions to a subfamily
		preserve each criterion.}\label{tab:KP-applications}
	\begin{tabular}{@{}p{0.33\textwidth}p{0.19\textwidth}p{0.18\textwidth}p{0.21\textwidth}@{}}
		\toprule
		Activity & Polymer space & Control $b(S)$ & Verification\\
		\midrule
		$w_D(S)$, including $w_\varnothing$ & $\cP$ & $\lambda|S|$ & Lemma~\ref{lem:KP-base}\\[1.5mm]
		$\widetilde w_D(S):=e^{\lambda d m_L(S)}w_D(S)$ & $\cP_{\leq}$ & $\lambda|S|$ & Lemma~\ref{lem:KP-tilted}\\[1.5mm]
		$u^{|S|}w_D(S)$, $|u|\leq R_*$ & $\mathcal A\subseteq\cP$ & $\lambda|S|/2$ & Lemma~\ref{lem:KP-marker}\\[1.5mm]
		$u^{|S|}\widetilde w_D(S)$, $|u|\leq R_*$ & $\mathcal A\subseteq\cP_{\leq}$ & $\lambda|S|/2$ & Lemma~\ref{lem:KP-marker}\\
		\bottomrule
	\end{tabular}
\end{table}

\subsection{Cluster tails and the exceptional event}\label{subsec:tails}

	For an ordered polymer tuple $\Gamma=(S_1,\ldots,S_j)$, define its
	weight, total size, and support by
	\begin{equation}\label{eq:cluster-notation}
	\wt_D(\Gamma):=\frac{\phit(S_1,\ldots,S_j)}{j!}\prod_iw_D(S_i),
	\quad M(\Gamma):=\sum_i|S_i|,
	\quad \Omega(\Gamma):=\bigcup_iS_i.
	\end{equation}
	Thus $M(\Gamma)$ counts vertices with multiplicity and may exceed
	$|\Omega(\Gamma)|$.

\begin{lemma}[Anchored cluster tails]\label{lem:cluster-tails}
	Let $F\subseteq V$ with $|F|\leq4$. There exists a constant $c>0$ such that, for all
	sufficiently large $n$,
	\begin{align}
	\sum_{\substack{\Gamma:\ \Omega(\Gamma)\cap F\neq\varnothing\\
			S_i\in\cP_{\leq},\ M(\Gamma)\geq L}}
	|\wt_D(\Gamma)|
	&\leq4\tau e^{-cd-\lambda dL},\label{eq:small-cluster-tail}\\
	\sum_{\substack{\Gamma:\ \Omega(\Gamma)\cap F\neq\varnothing\\
			|S_i|>N/2\ \mathrm{for\ some}\ i}}
	|\wt_D(\Gamma)|
	&\leq\tau e^{-\lambda N/4-\lambda dL}.
	\label{eq:large-cluster-tail}
	\end{align}
	Each sum ranges over ordered tuples of every positive length.
\end{lemma}
\begin{proof}
	For a tuple of small polymers with $M(\Gamma)\geq L$, we have
	$\sum_i m_L(S_i)\geq L$. Therefore
	\[
	\prod_i|w_D(S_i)|\leq e^{-\lambda dL}
	\prod_i|\widetilde w_D(S_i)|.
	\]
	Every tuple meeting $F$ has at least one pair $(x,i)$ with $x\in F\cap S_i$.
	Mark such a pair, use symmetry of the Ursell coefficient, and apply
	\eqref{eq:rooted-bound} with the activities from
	Lemma~\ref{lem:KP-tilted}. The factor $j/j!=1/(j-1)!$ gives
	\begin{align*}
	&\sum_{j\geq1}\frac1{j!}
	\sum_{\substack{(S_1,\ldots,S_j)\in\cP_{\leq}^j\\
			\Omega(\Gamma)\cap F\neq\varnothing}}
	|\phit(S_1,\ldots,S_j)|\prod_{i=1}^{j}|\widetilde w_D(S_i)|\\
	&\hspace{15mm}\leq\sum_{x\in F}\sum_{\substack{S_1\in \cP_{\leq} \\S_1\ni x}}
	|\widetilde w_D(S_1)|e^{\lambda|S_1|}
	\leq |F|\tau e^{-cd}.
	\end{align*}
	This proves \eqref{eq:small-cluster-tail}.
	
		For \eqref{eq:large-cluster-tail}, we may first discard the requirement
		that $\Omega(\Gamma)$ meet $F$. Since
		\[
		\one_{\{\max_i|S_i|>N/2\}}
		\leq \sum_{i=1}^j\one_{\{|S_i|>N/2\}},
		\]
		symmetry of the Ursell coefficient gives
		\[
		\begin{aligned}
		&\sum_{\Gamma:\, |S_i|>N/2\ {\rm for\ some}\ i}
		|\wt_D(\Gamma)|\\
		&\qquad\leq
		\sum_{j\geq1}\frac{1}{(j-1)!}
		\sum_{\substack{S_1\in\cP\\ |S_1|>N/2}}
		|w_D(S_1)|
		\sum_{(S_2,\ldots,S_j)\in\cP^{j-1}}
		|\phit(S_1,\ldots,S_j)|
		\prod_{i=2}^j|w_D(S_i)|.
		\end{aligned}
		\]
		The rooted bound \eqref{eq:rooted-bound}, applied with the activities
		$w_D$ and $b(S)=\lambda|S|$ as verified in
		Lemma~\ref{lem:KP-base}, bounds the inner sum by
		$e^{\lambda|S_1|}$. Hence the last display is at most
		\[
		\sum_{\substack{S_1\in\cP\\ |S_1|>N/2}}
		|w_D(S_1)|e^{\lambda|S_1|},
		\]
		which is bounded by
		\[
		\tau e^{-\lambda N/4-\lambda dL}
		\]
		by Lemma~\ref{lem:large-activity}. This proves
		\eqref{eq:large-cluster-tail}.

\end{proof}

\begin{lemma}[Configurations without a giant component]\label{lem:exceptional}
	There is $c>0$ such that
	\begin{align}
	\PP_p(\cD\mid D\subseteq\omega)&\leq e^{-cN-\lambda dL},
	\label{eq:no-giant-prob}\\
	0\leq\log\Zhat_D-\log Y_D&\leq\tau e^{-cN-\lambda dL}
	\label{eq:no-giant-log}
	\end{align}
	for all sufficiently large $n$.
\end{lemma}
\begin{proof}
	On $\cD$, denote by $\{K_1,\ldots,K_j\}$ the vertex sets of the connected components of $\big(V,\omega\big)$. These sets form a partition
	of $V$ into members of $\cP_{\leq}$. The number of
	closed edges joining distinct components is $\frac12\sum_i|\partial K_i|$.
	As in \eqref{eq:large-block-bound}, restoring the forced Bernoulli factors
	and then dropping internal connectivity gives
	\[
	\PP_p(\cD\mid D\subseteq\omega)
	\leq p^{-2}\sum_{\substack{\{K_1,\ldots,K_j\}\text{ partition of }V\\
			K_i\in\cP_{\leq}}}
	\prod_i\theta^{|\partial K_i|}.
	\]
	For such a partition, $\sum_i|K_i|=N$ and $\sum_i m_L(K_i)\geq L$.
	Using \eqref{eq:rho} and removing the partition constraints, we obtain
	\begin{align*}
	\PP_p(\cD\mid D\subseteq\omega)
	&\leq p^{-2}e^{-2\lambda N-2\lambda dL}
	\prod_{K\in\cP_{\leq}}(1+\rho(K))\\
	&\leq p^{-2}\exp\{-2\lambda N-2\lambda dL+Ne^{-c_1d}\}\\
	&\leq e^{-\lambda N-\lambda dL}
	\end{align*}
	for large $d$, by \eqref{eq:rho-sums} and $N\geq d+1$.
	Since $0\leq T_q\leq\one_\cD$, the same upper bound applies to $R_D$.

		Recall from
		\eqref{eq:log-Y-expansion} that $\log Y_D$ has an absolutely
		convergent cluster expansion. Marking one polymer in each tuple, using
		symmetry of the Ursell coefficient, and applying the rooted bound
		\eqref{eq:rooted-bound} with $b(S)=\lambda|S|$ gives
		\[
		|\log Y_D|
		\leq \sum_{S\in\cP}|w_D(S)|e^{\lambda|S|}
		\leq \sum_{x\in V}\sum_{\substack{S\in\cP\\S\ni x}}
		|w_D(S)|e^{\lambda|S|}
		\leq N\tau e^{-c_2d},
		\]
		where the last inequality follows from Lemma~\ref{lem:KP-base}. Hence
		\[
		Y_D\geq e^{-N\tau e^{-c_2d}}\geq e^{-Ne^{-c_2d}}.
		\]
	
	For large $d$, $e^{-c_2d}\leq\lambda/2$, so
	\[
	\frac{R_D}{Y_D}\leq e^{-\lambda N/2-\lambda dL}.
	\]
	Finally, \eqref{eq:Z-Y-R} and $0\leq\log(1+x)\leq x$ for $x\geq0$
	yield
	\[
	0\leq\log\Zhat_D-\log Y_D
	=\log\left(1+\tau\frac{R_D}{Y_D}\right)
	\leq\tau e^{-\lambda N/2-\lambda dL}.
	\]
	Both conclusions follow, for example, with $c=\lambda/2$ after increasing
	the threshold.
\end{proof}


\section{Cancellation and the shortest-path contribution}\label{sec:paths}
Fix distinct edges $e,f\in E$, and put $\ell=\dist_G(e,f)$. We shall
identify the first nonzero term in
\begin{equation}\label{eq:Lambda-Y}
\Lambda_Y(e,f):=\log Y_{\{e,f\}}+\log Y_\varnothing
-\log Y_{\{e\}}-\log Y_{\{f\}}.
\end{equation}
The sign of the individual cluster terms is not prescribed. The argument
instead computes their complete sum at minimum total size and bounds
all larger terms in absolute value.

\subsection{Cancellation and minimum total size}\label{subsec:cancellation}
\begin{lemma}\label{lem:cancellation}
	Let $\Gamma=(S_1,\ldots,S_j)\in\cP^j$ for some $j\ge1$. If  the  support $\Omega(\Gamma)=\bigcup_iS_i$ does not meet both $V(e)$ and $V(f)$, then it contributes
	zero to \eqref{eq:Lambda-Y}. If $\phit(S_1,\ldots,S_j)\neq0$, its
	support $\Omega(\Gamma)$ is connected in $G$.
\end{lemma}
\begin{proof}
		Suppose first that $\Omega(\Gamma)\cap V(e)=\varnothing$. Then $e$ is
		neither an internal nor a boundary edge of any polymer $S_i$. Hence
		\eqref{eq:activity} gives
		\[
		w_{\{e,f\}}(S_i)=w_{\{f\}}(S_i),
		\qquad
		w_{\{e\}}(S_i)=w_\varnothing(S_i)
		\]
		for every $i$. The corresponding products therefore cancel in the
		four-term difference defining $\Lambda_Y(e,f)$. The same argument, with
		$e$ and $f$ interchanged, applies when
		$\Omega(\Gamma)\cap V(f)=\varnothing$.
		
		Now suppose that $\phit(\Gamma)\neq0$. By the definition of the Ursell
		coefficient, the incompatibility graph $I_\Gamma$ is connected. Choose a
		spanning tree $T$ of $I_\Gamma$ and root it at an arbitrary index. If
		$\{i,k\}\in E(T)$, then $S_i\not\sim S_k$; hence $S_i$ and $S_k$ either
		intersect or are joined by an edge of $G$. Add the polymers in an order
		in which each parent precedes its children. Starting from the root
		polymer, each new polymer therefore intersects, or is joined by an edge
		of $G$ to, the union of the preceding polymers. Induction along $T$
		shows that
	$
		\Omega(\Gamma)=\bigcup_{i=1}^j S_i
	$
		is connected in $G$.
\end{proof}

	\begin{lemma}\label{lem:minimal-support}
		Suppose $\ell=\dist_G(e,f)\geq1$. If $U\subseteq V$ is connected and meets both
		$V(e)$ and $V(f)$, then
		\[
		|U|\geq\ell+1.
		\]
		Moreover, equality holds only if $G[U]$ is an induced shortest path from
		$V(e)$ to $V(f)$.
		
		Consequently, let $\Gamma=(S_1,\ldots,S_j)$ be a tuple giving a nonzero
		contribution to $\Lambda_Y(e,f)$, and set
		\[
		U:=\Omega(\Gamma)=\bigcup_{i=1}^j S_i.
		\]
		Then
		\begin{equation}\label{eq:min-size}
		M(\Gamma)\geq\ell+1.
		\end{equation}
		If equality holds in \eqref{eq:min-size}, then $U$ is the vertex set of
		an induced shortest path from $V(e)$ to $V(f)$, and the polymers
		$S_1,\ldots,S_j$ are pairwise disjoint intervals whose union is $U$.
		For $\ell=0$, \eqref{eq:min-size} also holds.
	\end{lemma}
	
	\begin{proof}
		Since $U$ is connected and meets both endpoint sets, $G[U]$ contains a
		path from $V(e)$ to $V(f)$. Such a path has at least $\ell$ edges and
		at most $|U|-1$ edges, whence $|U|\geq\ell+1$. If $|U|=\ell+1$, the
		path has exactly $\ell$ edges and visits every vertex of $U$. A chord
		would produce a shorter path between the endpoint sets, so $G[U]$ is
		an induced shortest path.
		
		Now let $\Gamma=(S_1,\ldots,S_j)$ give a nonzero contribution to
		$\Lambda_Y(e,f)$ and set $U=\Omega(\Gamma)$. By
		Lemma~\ref{lem:cancellation}, $U$ is connected and meets both $V(e)$
		and $V(f)$. Hence
		\[
		M(\Gamma)=\sum_{i=1}^j|S_i|
		\geq \left|\bigcup_{i=1}^jS_i\right|
		=|U|
		\geq\ell+1,
		\]
		which proves \eqref{eq:min-size}. If $M(\Gamma)=\ell+1$, then equality
		holds throughout the preceding display. Thus $|U|=\ell+1$, so the first
		part of the proof shows that $G[U]$ is an induced shortest path.
		Moreover,
		\[
		\sum_{i=1}^j|S_i|=\left|\bigcup_{i=1}^jS_i\right|,
		\]
		and hence the polymers $S_1,\ldots,S_j$ are pairwise disjoint. Since
		each $S_i$ is connected and contained in the path $G[U]$, it is an
		interval of that path. Their union is $U$, so these intervals partition
		the path.
		
		If $\ell=0$, every tuple contains at least one nonempty polymer, and
		therefore $M(\Gamma)\geq1=\ell+1$.
	\end{proof}

\subsection{Counting clusters of a given total size}\label{subsec:cluster-count}
Recall the incompatibility graph $I_\Gamma$ from Section~\ref{subsec:ueltschi}. Write
	$\mathcal T_j$ for the trees on the index set $[j]$, including the
	one-vertex tree when $j=1$.

\begin{lemma}[Tree-graph bound]\label{lem:tree-graph}
	For every ordered polymer tuple $\Gamma=(S_1,\ldots,S_j)$,
	\begin{equation}\label{eq:tree-graph}
	|\phit(S_1,\ldots,S_j)|
	\leq\sum_{T\in\mathcal T_j}
	\prod_{\{i,k\}\in E(T)}\one_{\{S_i\not\sim S_k\}}.
	\end{equation}
\end{lemma}
\begin{proof}
		If $I_\Gamma$ is disconnected, both sides of \eqref{eq:tree-graph}
		vanish. Suppose that $I_\Gamma$ is connected. By the definition of the
		Ursell coefficient,
		\[
		\phit(S_1,\ldots,S_j)
		=
		\sum_{\substack{A\subseteq E(I_\Gamma)\\([j],A)\ {\rm connected}}}
		(-1)^{|A|}.
		\]
		Writing $P_{I_\Gamma}(x)$ for the chromatic polynomial  of the graph $I_\Gamma$, the Whitney's spanning-subgraph expansion \cite{Whitney1932} gives
		\[
		P_{I_\Gamma}(x)
		=
		\sum_{A\subseteq E(I_\Gamma)}
		(-1)^{|A|}x^{k(A)},
		\]
		where $k(A)$ is the number of connected components of $([j],A)$.
		Writing $[x]P(x)$ for the coefficient of $x$ in a polynomial $P$, we
		therefore have
		\[
		\phit(S_1,\ldots,S_j)=[x]P_{I_\Gamma}(x).
		\]
		
		By Whitney's broken-circuit theorem \cite{Whitney1932}, the absolute
		value of this coefficient is the number of $(j-1)$-edge subsets of
		$E(I_\Gamma)$ containing no broken circuit. Every such edge set is
		acyclic and hence, having $j-1$ edges on $j$ vertices, is a spanning
		tree of $I_\Gamma$. Consequently,
		\[
		|\phit(S_1,\ldots,S_j)|
		\leq
		\#\{\text{spanning trees of }I_\Gamma\}
		=
		\sum_{T\in\mathcal T_j}
		\prod_{\{i,k\}\in E(T)}
		\one_{\{S_i\not\sim S_k\}},
		\]
		which proves \eqref{eq:tree-graph}. The case $j=1$ follows from the
		empty-product convention.
\end{proof}

\begin{lemma}\label{lem:cluster-count}
	There is an absolute constant $C_*>0$ such that, whenever
	$1\leq M\leq L\leq d$ and $F\subseteq V$ satisfies $1\leq|F|\leq4$,
	\begin{equation}\label{eq:cluster-count}
	\sum_{j=1}^{M}\frac1{j!}
	\sum_{\substack{(S_1,\ldots,S_j)\in\cP^j\\
			\sum_i|S_i|=M,\ \Omega(\Gamma)\cap F\neq\varnothing}}
	|\phit(S_1,\ldots,S_j)|\leq(C_*d^2)^M.
	\end{equation}
	One may take $C_*=32\mathrm{e}$, where $\mathrm{e}\approx 2.71828$ denotes Euler's number.
\end{lemma}
\begin{proof}
		Fix $j$, a positive composition
		\[
		m_1+\cdots+m_j=M,
		\]
		and a tree contributing to the right-hand side of
		\eqref{eq:tree-graph}. Choose a root index $r$ such that
		$S_r\cap F\neq\varnothing$, and root the tree at $r$; summing over all
		such choices only increases the count.
		
		For the root polymer $S_r$, choose a vertex $v\in S_r\cap F$. There are
		at most $|F|$ choices for $v$, and Lemma~\ref{lem:connected-set-count}
		gives at most $(4d)^{m_r-1}$ connected sets of size $m_r$ containing
		$v$. Hence the number of choices for $S_r$ is at most
		\[
		|F|(4d)^{m_r-1}.
		\]
		
		Now let $c$ be a child of $b$ in the rooted tree, and suppose that the
		parent polymer $S_b$ has already been chosen. The corresponding
		tree-graph factor requires $S_c\not\sim S_b$, and therefore
		\[
		S_c\cap\mathcal N(S_b)\neq\varnothing.
		\]
		Choose a vertex $v\in S_c\cap\mathcal N(S_b)$. Since
		\[
		|\mathcal N(S_b)|\leq(d+1)|S_b|=(d+1)m_b\leq2dM,
		\]
		there are at most
		\[
		2dM(4d)^{m_c-1}
		\]
		choices for the child polymer $S_c$. Proceeding away from the root, the
		number of polymer tuples associated with the fixed rooted tree and the
		fixed composition is therefore at most
		\[
		|F|(4d)^{M-j}(2dM)^{j-1}
		\leq4(4d^2)^{M-1},
		\]
		where $|F|\leq4$ and $M\leq d$ were used.
		
		By Cayley's formula, there are $j^{j-2}$ labeled trees on $[j]$ for
		$j\geq2$, and hence $j^{j-1}$ rooted labeled trees after choosing the
		root. For $j=1$, there is one rooted tree. There are
		$
		\binom{M-1}{j-1}
		$
		positive compositions of $M$ into $j$ parts. Moreover, the elementary
		factorial bound $\log(j!)\geq j\log j-j+1$ implies
		$j^{j-1}/j!\leq\mathrm {\mathrm{e}}^j$. 
		Summing over $j$ gives
		\[
		\begin{aligned}
		&\sum_{j=1}^M\frac1{j!}
		\sum_{\substack{(S_1,\ldots,S_j)\in\cP^j\\
				\sum_i|S_i|=M,\ \Omega(\Gamma)\cap F\neq\varnothing}}
		|\phit(S_1,\ldots,S_j)|\\
		&\qquad\leq
		4(4d^2)^{M-1}
		\sum_{j=1}^M
		\binom{M-1}{j-1}\mathrm{e}^j\\
		&\qquad=
		4e(4d^2)^{M-1}(1+\mathrm{e})^{M-1}\\
		&\qquad=
		4e\,[4(1+\mathrm{e})d^2]^{M-1}
		\leq(32\mathrm{e}d^2)^M.
		\end{aligned}
		\]
		Thus \eqref{eq:cluster-count} holds with $C_*=32\mathrm{e}$.
\end{proof}

\subsection{Internal edges and the small-cluster remainder}
\label{subsec:internal-edges}
The next estimate relates excess total size to the boundary penalty.
It also accounts for vertices appearing in more than one label.

\begin{lemma}\label{lem:internal-edges}
	Let $\Gamma=(S_1,\ldots,S_j)$ have nonzero Ursell coefficient, and set
	\[
	U=\Omega(\Gamma)=\bigcup_{i=1}^jS_i,
	\qquad
	M=M(\Gamma)=\sum_{i=1}^j|S_i|.
	\] Suppose that $U$ meets both $V(e)$ and $V(f)$ and that  $\ell+1\leq M\leq L$. Set
	\[
	k=|U|,\qquad \delta=k-(\ell+1),\qquad b=M-k,\qquad
	t=M-(\ell+1)=\delta+b.
	\]
	These quantities are nonnegative, and
	\begin{align}
	e(U)&\leq\ell+3\delta+\binom{\delta}{2},\label{eq:internal-U}\\
	\sum_i e(S_i)&\leq\ell+\frac{L+5}{2}\,t,\label{eq:internal-labels}\\
	\sum_i|\partial S_i|&\geq d(\ell+1+t)-2\ell-(L+5)t.
	\label{eq:boundary-excess}
	\end{align}
\end{lemma}
\begin{proof}
	Connectedness and Lemma~\ref{lem:minimal-support} imply $k\geq\ell+1$;
	clearly $M\geq k$. Choose a shortest path inside $G[U]$ joining the
	two endpoint sets, with length $m\geq\ell$, and put
	$\delta'=k-m-1\leq\delta$. The path is induced. A vertex outside it
	has at most three neighbors on it: two neighbors whose indices differ
	by at least three would yield a shorter connecting path. Therefore
	\[
	e(U)\leq m+3\delta'+\binom{\delta'}2
	=\ell+\delta+2\delta'+\binom{\delta'}2
	\leq\ell+3\delta+\binom\delta2.
	\]
	This argument also covers $\ell=0$: the shortest path may consist of
	the common endpoint alone.

		Order the polymers as $S_1,\ldots,S_j$, and set
		\[
		B_1:=\varnothing,
		\qquad
		B_i:=S_i\cap\bigcup_{q<i}S_q,\quad 2\leq i\leq j.
		\]
		Thus $B_i$ consists of the vertices of $S_i$ already covered by the preceding polymers. Since the sets
		\[
		S_i\setminus\bigcup_{q<i}S_q,\qquad 1\leq i\leq j,
		\]
		partition $U=\bigcup_iS_i$, we have
		\[
		k=|U|
		=\sum_{i=1}^j\left(|S_i|-|B_i|\right)
		=M-\sum_{i=1}^j|B_i|.
		\]
		Consequently,
		\[
		\sum_{i=1}^j|B_i|=M-k=b.
		\]
		
		The first occurrence of each edge in $\bigcup_i E(S_i)$ contributes at
		most $e(U)$. If an edge of $E(S_i)$ has already occurred in some
		$E(S_q)$ with $q<i$, then both of its endpoints belong to $B_i$.
		Consequently,
		\[
		\sum_{i=1}^j e(S_i)\leq e(U)+\sum_{i=1}^j e(B_i).
		\]
		Since $B_i\subseteq S_i$, the simplicity of $G$ and the bound
		$|S_i|\leq M\leq L$ give
		\[
		e(B_i)
		\leq \binom{|B_i|}{2}
		\leq \frac{|B_i|(|S_i|-1)}{2}
		\leq \frac{L-1}{2}|B_i|.
		\]
		Using $\sum_i|B_i|=b$, we conclude that
		\[
		\sum_{i=1}^j e(S_i)
		\leq e(U)+\frac{L-1}{2}b.
		\]

	Since $0\leq\delta\leq t\leq L-1$,
	$\binom\delta2\leq(L-1)\delta/2$ and $3\delta\leq3t$.
	Substituting \eqref{eq:internal-U} proves \eqref{eq:internal-labels}.
	Finally sum $|\partial S_i|=d|S_i|-2e(S_i)$ to obtain
	\eqref{eq:boundary-excess}.
\end{proof}

Recall $s=1-p$ and $a=-\log s$ from Section~\ref{subsec:conventions}. Put
\begin{equation}\label{eq:beta-eta}
\beta_d:=C_*d^2,\qquad \eta_L:=\beta_d s^{d-(L+5)}.
\end{equation}
Since $L=o(d)$,
\begin{equation}\label{eq:eta-small}
\log\eta_L=-ad+a(L+5)+2\log d+\log C_*=-ad+o(d),
\qquad 0<\eta_L\leq e^{-ad/2}\leq\tfrac12
\end{equation}
for all sufficiently large $n$.

\begin{lemma}\label{lem:small-remainder}
		For all sufficiently large $n$, uniformly in $q\in[0,1)$,
		$D\subseteq E$ with $|D|\leq2$, and distinct $e,f\in E$ satisfying
		$0\leq\ell=\dist_G(e,f)\leq r_n$,
		\begin{equation}\label{eq:small-remainder}
		\sum_{\substack{\Gamma:\,
				\Omega(\Gamma)\cap V(e)\neq\varnothing,\,
				\Omega(\Gamma)\cap V(f)\neq\varnothing\\
				\ell+2\leq M(\Gamma)\leq L}}
		|\wt_D(\Gamma)|
		\leq
		\tau\beta_d^{\ell+1}s^{d(\ell+1)-2\ell}
		\frac{\eta_L}{1-\eta_L}.
		\end{equation}
\end{lemma}
\begin{proof}
	Terms with zero Ursell coefficient may be discarded. Write
	$M=\ell+1+t$, where $1\leq t\leq L-\ell-1$. By
	Lemmas~\ref{lem:crude-activity} and~\ref{lem:internal-edges},
	\[
	\prod_i|w_D(S_i)|
	\leq\tau^j s^{\sum_i|\partial S_i|}
	\leq\tau s^{d(\ell+1+t)-2\ell-(L+5)t}.
	\]
	Apply Lemma~\ref{lem:cluster-count} with $F=V(e)\cup V(f)$ and sum in
	$t$. The resulting bound is
	\[
	\tau\beta_d^{\ell+1}s^{d(\ell+1)-2\ell}
	\sum_{t=1}^{L-\ell-1}(\beta_ds^{d-(L+5)})^t,
	\]
	which is at most \eqref{eq:small-remainder}. An empty size range has
	zero contribution.
\end{proof}

\subsection{Exact calculation on a shortest path}\label{subsec:path-calculation}
Let $\nu_\ell(e,f)$ be the number of sequences $(v_0,\ldots,v_\ell)$
with $v_0\in V(e)$, $v_\ell\in V(f)$, and consecutive vertices adjacent.
Because their length is $\dist_G(e,f)$, these sequences are shortest
paths. In particular,
\begin{equation}\label{eq:nu-positive}
\nu_\ell(e,f)\geq1,\qquad \nu_0(e,f)=1.
\end{equation}
A shortest path is induced and meets the endpoint sets only at its first
and last vertices. Its vertex set therefore determines its orientation
from $V(e)$ to $V(f)$, so different sequences have different vertex sets.
For $\ell=0$, the equality in \eqref{eq:nu-positive} uses simplicity of
$G$ and $e\neq f$.

For fixed $q\in[0,1)$, define
\begin{equation}\label{eq:z-alpha}
z:=(q-1)s^d,\qquad
\alpha:=\frac{p+(q-1)s}{(q-1)s^2}.
\end{equation}
The denominator is nonzero. We use $\alpha$ only in exact identities;
no estimate uniform in $q$ is made on this parameter.

\begin{lemma}\label{lem:tree-activity}
	If $1\leq|S|\leq N/2$ and $G[S]$ is a tree, then
	\begin{equation}\label{eq:tree-activity}
	w_\varnothing(S)=z^{|S|}\alpha^{e(S)}.
	\end{equation}
\end{lemma}
\begin{proof}
	Put $k=|S|$. For $H\subseteq E(S)$, one has
	$c(S,H)=k-|H|$, $e(S)=k-1$, and $\chi_-(S,H)=1$. The internal sum in
	\eqref{eq:activity} is thus
	\[
	\sum_{j=0}^{k-1}\binom{k-1}{j}(q-1)^{k-j}p^js^{k-1-j}
	=(q-1)[p+(q-1)s]^{k-1}.
	\]
	Multiplication by $s^{|\partial S|}=s^{dk-2(k-1)}$ gives
	\eqref{eq:tree-activity}.
\end{proof}

Let $P_m$ be the path with vertex set $[m]$ and edge set $E(P_m)=\{\{i,i+1\}:1\leq i<m\}$. 
For $I\subseteq[m]$,
let $e_{P_m}(I)$ denote the number of edges of $P_m$ with both endpoints
in $I$.  Define the polynomial
\begin{equation}\label{eq:Fm}
F_m(y):=\sum_{I\subseteq[m]}y^{|I|}\alpha^{e_{P_m}(I)},
\qquad F_0(y):=1.
\end{equation}
The maximal intervals of $I$ form a compatible polymer family, and every
compatible interval family is obtained this way. By
Lemma~\ref{lem:tree-activity}, $F_m(z)$ is therefore the partition
function of interval polymers on an induced path of $m\leq N/2$ vertices.

\begin{lemma}\label{lem:path-recurrence}
	For $m\geq2$,
	\begin{align}
	F_m(y)&=(1+\alpha y)F_{m-1}(y)+(1-\alpha)yF_{m-2}(y),
	\label{eq:path-recurrence}\\
	F_{m-2}(y)F_m(y)-F_{m-1}(y)^2
	&=(\alpha-1)^{m-1}y^m.\label{eq:path-determinant}
	\end{align}
\end{lemma}
\begin{proof}
	Separate the subsets in \eqref{eq:Fm} according to whether $m$ is
	absent, present with $m-1$ absent, or present with $m-1$ present. Their
	contributions are respectively
	$F_{m-1}$, $yF_{m-2}$, and $\alpha y(F_{m-1}-F_{m-2})$,
	which proves \eqref{eq:path-recurrence}.
	
	Write $A=1+\alpha y$ and $B=(1-\alpha)y$. Using the recurrence twice,
	\[
	F_{m-2}F_m-F_{m-1}^2
	=-B(F_{m-3}F_{m-1}-F_{m-2}^2).
	\]
	The initial values $F_0=1$, $F_1=1+y$, and $F_2=1+2y+\alpha y^2$
	give $F_0F_2-F_1^2=(\alpha-1)y^2$. Induction proves
	\eqref{eq:path-determinant}.
\end{proof}

\begin{proposition}[The minimum-size contribution]\label{prop:path-leading}
	Suppose $\ell\geq1$, $\ell+1\leq N/2$, and
	$P=(v_0,\ldots,v_\ell)$ is a shortest path from $V(e)$ to $V(f)$.
Then the total contribution to $\Lambda_Y(e,f)$ of tuples satisfying
		$M(\Gamma)=\ell+1$ and $\Omega(\Gamma)=V(P)$  is given by
		\begin{equation}\label{eq:path-leading}
		\begin{aligned}
		&\sum_{j=1}^{\ell+1}
		\sum_{\substack{\Gamma=(S_1,\ldots,S_j)\in\cP^j\\
				M(\Gamma)=\ell+1,\ \Omega(\Gamma)=V(P)}}
		\Bigl[
		\wt_{\{e,f\}}(\Gamma)+\wt_\varnothing(\Gamma)\\
		&\hspace{45mm}
		-\wt_{\{e\}}(\Gamma)-\wt_{\{f\}}(\Gamma)
		\Bigr]\\
		&\qquad=
		(q-1)s^{d(\ell+1)}
		\left(\frac{p(p+qs)}{s^2}\right)^\ell .
		\end{aligned}
		\end{equation}
\end{proposition}
\begin{proof}
	Set $m=\ell+1$ and let $\cP_P$ be the nonempty intervals of $P$.
	For $D\in\{\varnothing,\{e\},\{f\},\{e,f\}\}$, introduce
	\[
	Y_D^P(u):=\sum_{\substack{\mathscr S\subseteq\cP_P\\
			\mathscr S\ \mathrm{pairwise\ compatible}}}
	\prod_{S\in\mathscr S}u^{|S|}w_D(S).
	\]
	Lemma~\ref{lem:KP-marker}, with $v=w_D$ and
	$\mathcal A=\cP_P$, verifies the hypotheses of Ueltschi's theorem for
	these size-marked activities whenever $|u|\leq e^{\lambda/2}$ and $n$
	is sufficiently large. With logarithm zero at $u=0$,
	\begin{equation}\label{eq:path-marked-expansion}
	\log Y_D^P(u)=\sum_{j\geq1}\sum_{\Gamma\in\cP_P^j}
	u^{M(\Gamma)}\wt_D(\Gamma).
	\end{equation}
	For the coefficient calculation below, this identity also holds formally
	without any large-$n$ restriction, by Remark~\ref{rem:formal-marker}.
	Each coefficient involves finitely many nonempty labels. Thus the
	proposition itself only needs the stated size assumption.
	
	Exactly one endpoint of $e$ lies on $P$, namely $v_0$. An interval
	containing $v_0$ has $e$ in its boundary, so forcing $e$ open sets its
	activity to zero. All other interval activities are unchanged. Forcing
	$f$ has the analogous effect at $v_\ell$. Lemma~\ref{lem:tree-activity}
	and the interval representation of $F_m$ give
	\begin{equation}\label{eq:path-marked-F}
	Y_\varnothing^P(u)=F_m(uz),\qquad
	Y_{\{e\}}^P(u)=Y_{\{f\}}^P(u)=F_{m-1}(uz),\qquad
	Y_{\{e,f\}}^P(u)=F_{m-2}(uz).
	\end{equation}
	Define the marked logarithmic difference by
	\[
	\Lambda_Y^P(u)
	:=
	\log Y_{\{e,f\}}^P(u)+\log Y_\varnothing^P(u)
	-\log Y_{\{e\}}^P(u)-\log Y_{\{f\}}^P(u).
	\]
	By \eqref{eq:path-marked-F},
	\begin{align}
	\Lambda_Y^P(u)
	&=\log\frac{F_{m-2}(uz)F_m(uz)}{F_{m-1}(uz)^2}\notag\\
	&=\log\left(1+
	\frac{(\alpha-1)^{m-1}(uz)^m}{F_{m-1}(uz)^2}\right),
	\label{eq:path-log-ratio}
	\end{align}
	where the second equality follows from \eqref{eq:path-determinant}. Since $F_{m-1}(0)=1$, coefficient
	extraction gives
	\begin{equation}\label{eq:path-coefficient}
	[u^m]\Lambda_Y^P(u)=z^m(\alpha-1)^{m-1}.
	\end{equation}
	Indeed, the expression added to $1$ in \eqref{eq:path-log-ratio} starts
	at degree $m$, and its higher powers have degree greater than $m$.
	
	On the other hand, \eqref{eq:path-marked-expansion} shows that
	$[u^m]\Lambda_Y^P$ is the four-term sum over tuples of total size $m$.
	Lemma~\ref{lem:cancellation} cancels every tuple not touching both
	endpoints of $P$. Any remaining tuple with nonzero Ursell coefficient
	has connected support, and hence its support is all of $P$. Thus
	\eqref{eq:path-coefficient} is exactly the contribution specified in
	the proposition, without a multiplicity factor.
	
	Finally,
	\[
	\alpha-1=\frac{p+(q-1)s-(q-1)s^2}{(q-1)s^2}
	=\frac{p(p+qs)}{(q-1)s^2}.
	\]
	Substituting this and $z=(q-1)s^d$ into
	$z^{\ell+1}(\alpha-1)^\ell$ gives \eqref{eq:path-leading}.
\end{proof}


\section{Proof of the main theorem and graph examples}\label{sec:proof-examples}
\subsection{An asymptotic formula for the logarithmic correlation}
\label{subsec:asymptotic}
Recall $\Lhat_q$ from \eqref{eq:log-correlation} and $\Lambda_Y(e,f)$ from \eqref{eq:Lambda-Y}.  Also recall \eqref{eq:beta-eta}:
\[
\beta_d:=C_*d^2,
\qquad
\eta_L:=\beta_d s^{d-(L+5)}.
\] Put
\begin{equation}\label{eq:kappa}
\kappa_q:=\frac{p(p+qs)}{s^2},\qquad
\frac{p^2}{s^2}\leq\kappa_q\leq\frac{p}{s^2}.
\end{equation}
In particular, there is $C_p<\infty$, depending only on $p$, such that
$|\log\kappa_q|\leq C_p$ for every $q\in[0,1)$.

\begin{proposition}\label{prop:asymptotic}
	There are constants $c,C>0$ such that, for all sufficiently large $n$,
	every $q\in[0,1)$, and all distinct edges with
	$0\leq\ell=\dist(e,f)\leq r_n$,
	\begin{equation}\label{eq:asymptotic}
	\Lhat_q(e,f)
	=\nu_\ell(e,f)(q-1)s^{d(\ell+1)}\kappa_q^\ell+E_q(e,f),
	\end{equation}
	where
	\begin{equation}\label{eq:error-bound}
	|E_q(e,f)|\leq C\tau\left[
	\beta_d^{\ell+1}s^{d(\ell+1)-2\ell}\frac{\eta_L}{1-\eta_L}
	+e^{-cd-\lambda dL}\right].
	\end{equation}
	All constants and thresholds are independent of $q$ and the edge pair.
\end{proposition}
\begin{proof}
		Lemma~\ref{lem:exceptional}, applied to the four forcing sets, gives
		\begin{equation}\label{eq:Lambda-comparison}
		|\Lhat_q(e,f)-\Lambda_Y(e,f)|
		\leq4\tau e^{-c_1N-\lambda dL}.
		\end{equation}
		
		By \eqref{eq:log-Y-expansion}, the four logarithms defining
		$\Lambda_Y(e,f)$ have absolutely convergent cluster expansions. Hence
		\[
		\Lambda_Y(e,f)
		=
		\sum_{j\geq1}\sum_{\Gamma\in\cP^j}
		\Bigl[
		\wt_{\{e,f\}}(\Gamma)+\wt_\varnothing(\Gamma)
		-\wt_{\{e\}}(\Gamma)-\wt_{\{f\}}(\Gamma)
		\Bigr],
		\]
		with absolute convergence. By Lemma~\ref{lem:cancellation}, the bracket
		vanishes unless $\Omega(\Gamma)$ meets both $V(e)$ and $V(f)$.
		We partition the remaining tuples into three disjoint classes:
		\begin{enumerate}
			\item at least one polymer has size greater than $N/2$;
			\item all polymers have size at most $N/2$, and $M(\Gamma)\geq L+1$;
			\item all polymers have size at most $N/2$, and $M(\Gamma)\leq L$.
		\end{enumerate}
		
		For each of the four forcing sets, Lemma~\ref{lem:cluster-tails} bounds
		the first two classes by $\tau e^{-c_2N-\lambda dL}$ and
		$4\tau e^{-c_3d-\lambda dL}$, respectively. In the second class,
		cancellation ensures that the support meets
		\[
		F:=V(e)\cup V(f),
		\qquad |F|\leq4,
		\]
		so the anchored estimate in Lemma~\ref{lem:cluster-tails} applies.
		
		It remains to consider the third class. By
		Lemma~\ref{lem:minimal-support},
		\[
		M(\Gamma)\geq\ell+1.
		\]
		Suppose first that $\ell\geq1$. If $M(\Gamma)=\ell+1$, then
		Lemma~\ref{lem:minimal-support} shows that $\Omega(\Gamma)$ is the
		vertex set of an induced shortest path from $V(e)$ to $V(f)$, and the
		polymers in $\Gamma$ form a partition of that path into intervals.
		Moreover,
		\[
		\ell+1\leq r_n+1<L<N/2,
		\]
		so all these polymers lie in the small-polymer regime.
		Proposition~\ref{prop:path-leading} gives
		\[
		(q-1)s^{d(\ell+1)}\kappa_q^\ell
		\]
		for each shortest path. Since distinct shortest paths have distinct
		vertex sets and there are $\nu_\ell(e,f)$ of them, the total
		contribution from $M(\Gamma)=\ell+1$ is
		\[
		\nu_\ell(e,f)(q-1)s^{d(\ell+1)}\kappa_q^\ell,
		\]
		which is the leading term in \eqref{eq:asymptotic}.
		
		If $\ell=0$, let $v$ be the common endpoint of $e$ and $f$. The only
		contributing tuple with $M(\Gamma)=1$ is the singleton polymer $\{v\}$.
		Its activities are
		\[
		w_\varnothing(\{v\})=(q-1)s^d,\qquad
		w_{\{e\}}(\{v\})=w_{\{f\}}(\{v\})
		=w_{\{e,f\}}(\{v\})=0.
		\]
		Thus the same leading term is obtained, since
		$\nu_0(e,f)=1$ and $\kappa_q^0=1$.
		
		All remaining tuples in the third class satisfy
		\[
		\ell+2\leq M(\Gamma)\leq L.
		\]
		By \eqref{eq:small-remainder}, their total absolute contribution,
		summed over the four forcing sets, is at most
		\[
		4\tau\beta_d^{\ell+1}s^{d(\ell+1)-2\ell}
		\frac{\eta_L}{1-\eta_L}.
		\]
		
		Combining the three classes with \eqref{eq:Lambda-comparison}, and using
		$N\geq d+1$, we obtain, after decreasing $c>0$ and increasing $C<\infty$
		if necessary,
		\[
		|E_q(e,f)|
		\leq
		C\tau\left[
		\beta_d^{\ell+1}s^{d(\ell+1)-2\ell}
		\frac{\eta_L}{1-\eta_L}
		+
		e^{-cd-\lambda dL}
		\right].
		\]
		This proves \eqref{eq:asymptotic} and \eqref{eq:error-bound}.
\end{proof}

\subsection{Uniform comparison and strict negativity}\label{subsec:main-proof}
\begin{proof}[Proof of Theorem~\ref{thm:main}]
	Fix $q\in[0,1)$ and distinct edges $e,f\in E$ with
	\[
	0\leq\ell:=\dist_G(e,f)\leq r_n.
	\]
	Define
	\[
	\mathfrak M_{n,q,e,f}
	:=\tau\nu_\ell(e,f)s^{d(\ell+1)}\kappa_q^\ell>0,
	\qquad \ell=\dist(e,f).
	\]
	The leading term in \eqref{eq:asymptotic} is
	$-\mathfrak M_{n,q,e,f}$. We show that the error divided by this
	quantity tends to zero uniformly over all admissible triples.
	
	For the first term in \eqref{eq:error-bound}, use
	$\nu_\ell\geq1$, $|\log\kappa_q|\leq C_p$, and $\eta_L\leq1/2$.
	The logarithm of the ratio to $\mathfrak M_{n,q,e,f}$ is at most
	\begin{align}
	&\log(2C)+(\ell+1)\log\beta_d+2a\ell+C_p\ell+\log\eta_L
	\notag\\
	&\quad\leq-ad+a(L+5)+(\ell+2)(\log C_*+2\log d)
	+(2a+C_p)\ell+\log(2C)
	\notag\\
	&\quad\leq-ad+K_{p,h}(r_n+1)\log d,
	\label{eq:relative-small-log}
	\end{align}
	where $K_{p,h}$ is independent of the triple. Since
	$(r_n+1)\log d=o(d)$, the last bound is at most $-ad/2$ eventually.
	
	For the second term, the logarithm of the relative bound is at most
	\begin{equation}\label{eq:relative-tail-log}
	\log C-cd-\lambda dL+ad(\ell+1)+C_p\ell.
	\end{equation}
	The definitions of $\lambda$ and $L$ give
	\[
	\lambda L\geq\frac{ah}{8}\frac{16}{h}(r_n+2)
	=2a(r_n+2).
	\]
	Thus $-\lambda dL+ad(\ell+1)\leq-ad(r_n+3)\leq 0$, and then
	\eqref{eq:relative-tail-log} is at most
	$\log C-cd+C_pr_n$. Because $r_n=o(d)$, this is at most $-cd/2$
	for all sufficiently large $n$. We have proved
	\begin{equation}\label{eq:relative-error}
	\sup_{\substack{q\in[0,1),\ e\neq f\\\dist_{G_n}(e,f)\leq r_n}}
	\frac{|E_q(e,f)|}{\mathfrak M_{n,q,e,f}}\longrightarrow0.
	\end{equation}
	The common factor $\tau=1-q$ cancels in both comparisons. No lower
	bound on $1-q$ has been imposed.
	
	For sufficiently large $n$, the supremum in \eqref{eq:relative-error}
	is less than $1/2$, so $\Lhat_q(e,f)<0$. The exact sign identity
	\eqref{eq:sign-equivalence} proves the theorem. The same argument gives
	the relative asymptotic formula \eqref{eq:intro-asymptotic}.
\end{proof}

Every estimate in this proof depends on the graph only through its
regular degree, its number of vertices, and the common isoperimetric
parameters. In particular, the threshold can be chosen uniformly over
all graphs satisfying these numerical hypotheses. This observation
will be used for the random-regular example below.

\subsection{Uniform normalized expansion and deterministic examples}
\label{subsec:spectral-examples}
For a finite $d$-regular graph with $d\geq1$, define the normalized edge Cheeger constant
\begin{equation}\label{eq:normalized-expansion}
h_E(G):=\min_{1\leq|S|\leq N/2}\frac{|\partial_GS|}{d|S|}.
\end{equation}

\begin{corollary}\label{cor:expanders}
	Suppose $d_n\to\infty$ and $h_E(G_n)\geq h_0>0$ for all sufficiently
	large $n$. Then the conclusion of Theorem~\ref{thm:main} holds for every
	sequence $r_n\log d_n=o(d_n)$.
\end{corollary}
\begin{proof}
	Choose $h=h_0/2$. Since $\min\{|S|,L_n\}\leq|S|$ and
	$d_n+1\leq2d_n$,
	\[
	h\bigl(|S|+d_n\min\{|S|,L_n\}\bigr)
	\leq h_0d_n|S|\leq|\partial_{G_n}S|.
	\]
	Thus the graph condition holds for every value of $L_n$.
\end{proof}

The following elementary spectral criterion supplies concrete examples.
Write the eigenvalues of the adjacency matrix of the $d$-regular graph
$G$ as
\[
d=\lambda_1(G)\geq\lambda_2(G)\geq\cdots\geq\lambda_N(G).
\]

\begin{lemma}\label{lem:spectral-cut}
	For $S\subseteq V$ of size $k$,
	\begin{equation}\label{eq:spectral-cut}
	|\partial_GS|\geq(d-\lambda_2(G))k(1-k/N).
	\end{equation}
	Consequently, $h_E(G)\geq(d-\lambda_2(G))/(2d)$.
\end{lemma}
\begin{proof}

		This is the standard spectral bound for edge expansion; see
		\cite[Lemma~2.1]{AlonMilman1985} and the discussion of regular graphs
		in Section~3 therein.

\end{proof}

This is a one-sided eigenvalue criterion. Replacing $\lambda_2$ by the
largest absolute nonconstant eigenvalue would lose the useful bound
for bipartite graphs. In particular, any high-degree family with
$\lambda_2(G_n)\leq(1-\varepsilon)d_n$, for fixed $\varepsilon>0$,
satisfies Corollary~\ref{cor:expanders}. The following families illustrate
the range of the hypothesis.

\paragraph{\bfseries Complete and balanced complete multipartite graphs.}
For $K_N$, $d=N-1$ and
$|\partial S|=|S|(N-|S|)\geq(d+1)|S|/2$ when $|S|\leq N/2$.
Thus \eqref{eq:two-scale-isoperimetry} holds directly with $h=1/2$.
More generally, fix an integer $r\geq2$ and consider the balanced
complete $r$-partite graph $K_{m,\ldots,m}$, with $r$ parts of size
$m\ge2$. Ordering the vertices by parts, its adjacency matrix is
\[
A=(J_r-I_r)\otimes J_m,
\]
where $J_t$ denotes the $t\times t$ all-ones matrix and $\otimes$ denotes Kronecker product. The standard spectral formula for Kronecker products
\cite[Section~4.2]{HornJohnson1991} gives eigenvalues
$(r-1)m$, $-m$ with multiplicity $r-1$, and $0$ with multiplicity
$r(m-1)$. Thus
$d=(r-1)m$ and $\lambda_2(G)=0$ for $m\geq2$, so
Lemma~\ref{lem:spectral-cut} gives $h_E(G)\geq1/2$.
This includes $K_{m,m}$. The degrees diverge as $m\to\infty$, while
the diameters are at most two.

\paragraph{\bfseries Fixed-dimensional Hamming graphs with growing alphabet.}
Fix $k\geq2$ and let
\[
H(k,r):=\underbrace{K_r\square\cdots\square K_r}_{k\text{ factors}}
\]
be the Hamming graph; see \cite[Section~9.2]{BrouwerCohenNeumaier1989}.
Its vertex set is $[r]^k$, with two vertices adjacent if they differ
in exactly one coordinate. The graph is $d=k(r-1)$ regular. The adjacency matrix of a Cartesian product is the Kronecker sum of the adjacency matrices of its factors; see \cite[Section~4.4]{HornJohnson1991}. Hence its eigenvalues are the corresponding sums of factor eigenvalues. Since $K_r$ has eigenvalues $r-1$ and $-1$, it follows that
\[
\lambda_2(H(k,r))=(k-1)(r-1)-1=d-r.
\]
Thus
\[
h_E(H(k,r))
\geq\frac{d-\lambda_2(H(k,r))}{2d}
=\frac{r}{2k(r-1)}
\geq\frac{1}{2k}.
\]
For fixed $k$, the degree tends to infinity with $r$, while the diameter is $k$.
This regime
is distinct from the fixed-alphabet, growing-dimension products treated
in the next subsection.

\paragraph{\bfseries Paley graphs.}
Let $r\equiv1\pmod4$ be a prime power, and let
$\operatorname{Pal}(r)$ be the Paley graph on $\mathbb F_r$: distinct
$x,y\in\mathbb F_r$ are adjacent if $x-y$ is a square in
$\mathbb F_r^\times$; see, for example,
\cite[pp.~221--222]{GodsilRoyle2001}. Since $-1$ is a square in
$\mathbb F_r$, this relation is symmetric. The graph is
$d_r=(r-1)/2$ regular. Moreover, two distinct vertices have
$(r-5)/4$ common neighbors if they are adjacent and $(r-1)/4$ common
neighbors otherwise. Equivalently, if $A$ is the adjacency matrix of
$\operatorname{Pal}(r)$, then
\[
A^2+A=\frac{r-1}{4}(I+J).
\]
Since $J$ vanishes on $\one^\perp$, every eigenvalue of $A$ on
$\one^\perp$ satisfies
\[
\lambda^2+\lambda-\frac{r-1}{4}=0.
\]
Thus the adjacency eigenvalues other than $d_r$ are
\[
\frac{-1+\sqrt r}{2}
\qquad\text{and}\qquad
\frac{-1-\sqrt r}{2}.
\]
Consequently, Lemma~\ref{lem:spectral-cut} gives
\begin{equation}\label{eq:Paley-expansion}
h_E(\operatorname{Pal}(r))
\geq\frac{r-\sqrt r}{2(r-1)}
\longrightarrow\frac12.
\end{equation}
Finally, since any two nonadjacent vertices have $(r-1)/4>0$ common
neighbors, $\operatorname{Pal}(r)$ has diameter two. Thus this family
also yields an all-pairs application.

\paragraph{\bfseries Blow-ups of a fixed regular graph.}
Fix a connected $\Delta$-regular graph $H$ with $\Delta\geq1$. For
$m\geq1$, let
\[
H^{(m)}:=H[\overline K_m],
\]
the lexicographic product of $H$ with the edgeless graph on $m$
vertices; see \cite[Definition~2]{Sabidussi1959}. Equivalently,
$H^{(m)}$ is obtained by replacing each vertex $v\in V(H)$ by an
independent set $C_v$ of size $m$ and joining $C_u$ and $C_v$
completely whenever $uv\in E(H)$.
Its adjacency matrix is
\[
A_{H^{(m)}}=A_H\otimes J_m.
\]
Since $J_m$ has eigenvalues $m$ and $0$, the standard spectral formula
for Kronecker products \cite[Section~4.2]{HornJohnson1991} shows that
$H^{(m)}$ is $\Delta m$ regular and, for $m\geq2$,
\[
\lambda_2(H^{(m)})=m\max\{\lambda_2(H),0\}.
\]
Consequently, Lemma~\ref{lem:spectral-cut} gives
\[
h_E(H^{(m)})
\geq
\frac{\Delta-\max\{\lambda_2(H),0\}}{2\Delta}>0.
\]
The strict inequality follows from $\lambda_2(H)<\Delta$, since $H$ is
connected and $\Delta$-regular. The assumption $\Delta\geq1$ excludes
the one-vertex graph, whose blow-ups have degree zero. Moreover,
\[
\diam(H^{(m)})\leq\max\{2,\diam(H)\}.
\]

Each of the preceding deterministic families has diverging degree and
uniformly bounded diameter. Taking $r_n$ to be a common diameter bound,
Theorem~\ref{thm:main} therefore yields strictly negative covariance for
every distinct edge pair in all sufficiently large members of each
family, uniformly over $q\in[0,1)$.

\subsection{Cartesian powers: hypercubes and fixed-side tori}
\label{subsec:product-examples}

We now carry out the product-graph verification announced in Section~\ref{subsec:examples-intro}. Although these Cartesian
powers have normalized expansion tending to zero, their small-set expansion remains of degree scale;
the weaker large-set part of the two-scale hypothesis then supplies the control needed at macroscopic
scales. The following formulation includes the two principal examples and uses the same proof for
any fixed regular factor.

\begin{corollary}\label{cor:products}
	Let $B$ be a fixed finite connected simple $\Delta$-regular graph with
	$b\geq2$ vertices and $\Delta\geq1$. Let $G_n=B^{\square n}$ be the $n$-fold Cartesian power of $B$, so that
	$N_n=b^n$ and $d_n=n\Delta$. For every sequence of nonnegative integers
	$r_n$ satisfying $r_n\log d_n=o(d_n)$, the conclusion of
	Theorem~\ref{thm:main} holds on $G_n$.
\end{corollary}
\begin{proof}
	For $1\leq k\leq b-1$, put
	\[
	i_k(B):=\min_{|A|=k}\frac{|\partial_BA|}{k},\qquad
	y_B:=\min_{1\leq k\leq b-1}
	\frac{i_k(B)\log b}{\log b-\log k}>0.
	\]
	The positivity follows from connectedness of the fixed graph $B$.
	Diskin and Samotij's product inequality gives
	\begin{equation}\label{eq:product-profile}
	|\partial_{G_n}S|\geq y_B|S|(n-\log_b|S|)
	\qquad(\varnothing\neq S\subseteq V(G_n));
	\end{equation}
	see \cite[Section~3.5, equation~(8)]{DiskinSamotij2025}.
	Let $\sigma=\log_b2$ and choose
	\begin{equation}\label{eq:product-h}
	h:=\min\left\{\frac{y_B}{4\Delta},\frac{y_B\sigma}{2}\right\}.
	\end{equation}
	With $L=L_n$ from \eqref{eq:L-definition}, the distance assumption
	implies $L=o(n)$. Thus, for all sufficiently large $n$,
	\[
	L<b^n/2,\qquad \log_bL\leq n/2,
	\qquad \sigma b^n\geq nL.
	\]
	
	Write $u=|S|\leq b^n/2$. If $u\leq L$, then
	\eqref{eq:product-profile} gives
	$|\partial S|\geq y_Bnu/2\geq h(1+n\Delta)u$,
	using $1+n\Delta\leq2n\Delta$ and \eqref{eq:product-h}.
	If $L\leq u\leq b^n/2$, the function
	$g(t)=t(n-\log_bt)$ is concave, since $g''(t)=-1/(t\log b)<0$.
	It satisfies
	\[
	g(u)\geq\sigma u,
	\qquad g(u)\geq\min\{g(L),g(b^n/2)\}\geq nL/2.
	\]
	Taking half of each bound yields
	\[
	|\partial S|\geq y_Bg(u)
	\geq\frac{y_B\sigma}{2}u+\frac{y_B}{4}nL
	\geq h(u+n\Delta L).
	\]
	The two ranges verify \eqref{eq:two-scale-isoperimetry}, completing
	the proof.
\end{proof}

For $B=K_2$, the graph is the hypercube $\mathsf Q_n$, with
$d_n=n$, $y_B=1$, and the familiar bound
\[
|\partial_{\mathsf Q_n}S|\geq|S|(n-\log_2|S|);
\]
see also \cite{Harper1964,Hart1976}. For $B=C_m$, where $m\geq3$ is
fixed, the graph is the discrete torus $(\mathbb Z/m\mathbb Z)^n$
with nearest-neighbor edges and degree $2n$. A nonempty proper subset
of the cycle has at least two boundary edges, attained by an interval.
Therefore
\[
y_{C_m}=\min_{1\leq k\leq m-1}
\frac{2\log m}{k(\log m-\log k)}>0.
\]
These are precisely the fixed-side tori, rather than tori whose side
length varies with $n$.

For any fixed factor $B$, a coordinate slice
$S=A\times V(B)^{n-1}$ with $0<|A|\leq b/2$ has
\[
\frac{|\partial_{G_n}S|}{d_n|S|}
=\frac{|\partial_BA|}{n\Delta|A|}\longrightarrow0.
\]
This recovers the obstruction described in Section~\ref{subsec:examples-intro}: uniform normalized expansion excludes these
products, whereas the two-scale condition retains them by weakening the requirement beyond the
cutoff $L_n$.
\subsection{Uniformly random simple regular graphs}\label{subsec:random-regular}
The edge expansion of random regular graphs has been studied extensively; see, in particular, Bollob{\'a}s~\cite{Bollobas1988} and Kolesnik and Wormald~\cite{KolesnikWormald2014}. For our application, we need a quantitative bound that is uniform over arbitrary diverging degree sequences. We obtain the required form from available spectral estimates.

For integers $N\geq1$ and $0\leq d\leq N-1$ with $Nd$ even, let
\[
\Omega_{N,d}:=\{G: G\text{ is a simple }d\text{-regular graph on }[N]\}.
\]
These feasibility conditions ensure $\Omega_{N,d}\neq\varnothing$.
Equip it with the sigma-algebra $2^{\Omega_{N,d}}$ and the uniform law
\begin{equation}\label{eq:uniform-regular-law}
\mathbb P_{N,d}(\mathcal B)=\frac{|\mathcal B|}{|\Omega_{N,d}|},
\qquad \mathcal B\subseteq\Omega_{N,d}.
\end{equation}
The vertices are fixed and labeled; only the edge set is random.

\begin{proposition}\label{prop:random-iso}
	Let $(N_n,d_n)$ satisfy the feasibility conditions above and suppose that
	$d_n\to\infty$.  Write
	$\mathbb P_n=\mathbb P_{N_n,d_n}$. Then, for every fixed $0<h<1/2$,
\begin{equation}
\label{eq:random-reg-expander}
\mathbb P_n\big[h_E(G_n)\geq h \big]\geq 1-\frac{1}{N_n}
\end{equation}
for all sufficiently large $n$.
\end{proposition}
\begin{proof}
	Write $N=N_n,d=d_n$, and let
	\[
	\lambda_{\rm abs}(G):=\|A_G|_{\one^\perp}\|
	=\max_{2\leq i\leq N}|\lambda_i(G)|.
	\]
	This definition applies also to disconnected graphs. We first derive a spectral estimate that is uniform over the full degree range: 
	there is an absolute $C$ such that, uniformly over feasible $d$ and
	all sufficiently large $N$,
	\begin{equation}\label{eq:random-spectral}
	\mathbb P_{N,d}\{\lambda_{\rm abs}(G)>1+C\sqrt d\}\leq N^{-1}.
	\end{equation}
	Cook, Goldstein, and Johnson's Theorem~1.1 in \cite{CookGoldsteinJohnson2018}, with $C_0=K=1$, gives
	\[
	\mathbb P_{N,d}\{\lambda_{\rm abs}(G)>C_1\sqrt d\}\leq N^{-1}
	\qquad(1\leq d\leq N^{2/3})
	\]
	for all sufficiently large $N$.
	Tikhomirov and Youssef's Theorem~A in \cite{TikhomirovYoussef2019}, with $\alpha=1/2$ and $m=1$,
	gives the same bound with a constant $C_2$ when
	$N^{1/2}\leq d\leq N/2$ .
	The constants and lower thresholds on $N$ are independent of $d$ in
	the respective ranges. These ranges overlap and cover $1\leq d\leq N/2$.
	Taking $C=\max\{C_1,C_2,1\}$ yields the bound with $C\sqrt d$
	throughout that range. The case $d=0$ is deterministic.
	
	If $d>N/2$, write $a_0=N-1-d$ and let $\overline G$ be the complement.
	Complementation maps the uniform law on $\Omega_{N,d}$ to that on
	$\Omega_{N,a_0}$. Since
	$A_G+A_{\overline G}=J-I$, where $J$ is again the all-one square matrix, on $\one^\perp$ we have
	\[
	\lambda_{\rm abs}(G)\leq1+\lambda_{\rm abs}(\overline G).
	\]
	Here $a_0<N/2$ and $a_0<d$. The previously established estimate gives
	\eqref{eq:random-spectral} in this range as well.
	
	By Lemma~\ref{lem:spectral-cut} and 
	$\lambda_2(G)\leq\lambda_{\rm abs}(G)$,
	\[
	h_E(G)\geq \frac{d-\lambda_{\rm abs}(G)}{2d}
	\] 
	Hence on the complementary event  to $\{ \lambda_{\rm abs}(G)>1+C\sqrt d \, \}$, 
	\[
	h_E(G)\geq \frac{d-1-C\sqrt{d}}{2d}\,.
	\]
	Since $d_n\rightarrow \infty$, 
	\[
	\frac{d-1-C\sqrt{d}}{2d}\rightarrow \frac{1}{2}\,.
	\] 
	Hence $\frac{d-1-C\sqrt{d}}{2d}$  is at least $h$ for every fixed  $h\in(0,\frac{1}{2})$ and all sufficiently large $n$. Together with \eqref{eq:random-spectral}, this proves the proposition.
\end{proof}

 If the graphs $G_n$ are realized on
a common probability space with marginals $\mathbb P_n$ and
\[
\sum_{n=1}^\infty N_n^{-1}<\infty,
\]
then, for every fixed $h\in(0,\frac{1}{2})$, Proposition~\ref{prop:random-iso} and the first Borel--Cantelli lemma imply that
\[
h_E(G_n)\geq h
\]
eventually almost surely. No independence among the $(G_n)$ is required. 

\begin{corollary}\label{cor:random-NC}
	Fix $p\in(0,1)$ and feasible $(N_n,d_n)$ with $d_n\to\infty$.
	Let $r_n$ be nonnegative integers with $r_n\log d_n=o(d_n)$.
	For all sufficiently large $n$, with $\mathbb P_n$-probability at least
	$1-N_n^{-1}$ the graph $G$ is connected and satisfies
	\[
	\Cov_{\mu_{G,p,q}}(X_e,X_f)<0
	\quad\text{for every }q\in[0,1)
	\text{ and every }e\neq f\text{ with }\dist_G(e,f)\leq r_n.
	\]
\end{corollary}
\begin{proof}
	Fix $0<h<1/2$. By Proposition~\ref{prop:random-iso}, $h_E(G)\geq h$ with probability at least $1-N_n^{-1}$ for all sufficiently large $n$. On this event the random graph $G$ is connected, and the conclusion follows from Corollary~\ref{cor:expanders}.
\end{proof}

\subsection{An all-pairs criterion}\label{subsec:diameter}

The preceding deterministic examples in Section~\ref{subsec:spectral-examples} have uniformly bounded diameter. More generally, the distance restriction in Theorem~\ref{thm:main}
can be removed whenever the graph condition holds with
$r_n=\diam(G_n)$ and the diameter satisfies the required growth
condition.

\begin{corollary}\label{cor:diameter}
	Suppose that the graph condition of Theorem~\ref{thm:main} holds with
	$r_n=\diam(G_n)$ and the corresponding $L_n$. If
	\[
	\diam(G_n)\log d_n=o(d_n),
	\]
	then, for all sufficiently large $n$,
	\[
	\Cov_{\mu_{G_n,p,q}}(X_e,X_f)<0
	\]
	for every $q\in[0,1)$ and every pair of distinct edges
	$e,f\in E(G_n)$.
\end{corollary}

\begin{proof}
	For distinct edges $e,f\in E(G_n)$,
	\[
	\dist_{G_n}(e,f)\leq\diam(G_n)=r_n.
	\]
	Thus every edge pair satisfies the distance hypothesis of
	Theorem~\ref{thm:main}, and the conclusion follows.
\end{proof}

In particular, for each of the deterministic families considered in
Section~\ref{subsec:spectral-examples} with uniformly bounded diameter,
the all-pairs conclusion holds for all sufficiently large members of
the family.


\appendix
\section{Notation index}\label{sec:notation-index}
This index collects the notation used in the proof. In Sections~\ref{sec:estimates}
--\ref{sec:proof-examples}, the index $n$ is usually suppressed. The
location column points to a definition or a principal use. Constants
$c,C$ may change between statements and depend only on $p,h$ unless
another dependence is indicated.

\begingroup
\small
\setlength{\tabcolsep}{5pt}

\begin{longtable}{@{}>{\raggedright\arraybackslash}p{0.23\textwidth}>{\raggedright\arraybackslash}p{0.51\textwidth}>{\raggedright\arraybackslash}p{0.18\textwidth}@{}}
	\caption{Notation for the model, graph geometry, and cluster expansion.}
	\label{tab:notation-index}\\
	\toprule
	Symbol & Meaning & Location\\
	\midrule
	\endfirsthead
	\multicolumn{3}{@{}l}{\textit{Table~\ref{tab:notation-index} continued}}\\
	\toprule
	Symbol & Meaning & Location\\
	\midrule
	\endhead
	\midrule
	\multicolumn{3}{r@{}}{\textit{Continued on the next page}}\\
	\endfoot
	\bottomrule
	\endlastfoot
	\multicolumn{3}{@{}l}{\textbf{Graph and model parameters}}\\
	$G_n=(V_n,E_n)$ & Finite connected simple regular graph; $G=(V,E)$ when $n$ is fixed. & \S\ref{subsec:main}\\
	$J_{D,r}$, $R_D(H)$ & Rooted degree-correction gadget and regularization of $H$. In Section~\ref{sec:reduction}, $D$ is an odd target degree, not a forcing set. & Lemma~\ref{lem:gadget}, \eqref{eq:regularization}\\
	$N_n$, $d_n$ & Number of vertices and common degree; abbreviated $N,d$. & \S\ref{subsec:main}\\
	$E(S)$, $e(S)$ & Internal edge set of $S$ and its cardinality. & \S\ref{subsec:main}\\
	$\partial S$ & Edges with exactly one endpoint in $S$. & \S\ref{subsec:main}\\
	$\Nhood(S)$ & Closed vertex neighborhood of $S$ in $G$. & \eqref{eq:neighborhood}\\
	$V(e)$, $\dist_G(e,f)$ & Endpoint set of an edge; minimum vertex distance between two endpoint sets. & \eqref{eq:edge-distance}\\
	$r_n$, $h$, $L_n$ & Allowed edge distance, isoperimetric constant, and cutoff $\lceil16(r_n+2)/h\rceil$. & \eqref{eq:distance-assumption}--\eqref{eq:two-scale-isoperimetry}\\
	$m_L(S)$ & $\min\{|S|,L\}$. & \eqref{eq:L-definition}\\
	$p$, $s$, $a$ & Fixed edge parameter; $s=1-p$, $a=-\log s>0$. & \S\ref{subsec:conventions}\\
	$q$, $\tau$ & Cluster parameter in $[0,1)$; $\tau=1-q$. & \S\ref{subsec:giant}\\
	$\lambda$, $\theta$ & Fixed weights $\lambda=ah/8$, $\theta=s^{1/2}$. & \eqref{eq:lambda-theta}\\
	$\omega$, $k(\omega)$ & Open-edge configuration and number of its spanning components. & \S\ref{sec:introduction}\\
	$X_g$ & Indicator that edge $g$ is open. & \S\ref{subsec:negative-dependence}\\
	$\PP_p$, $\EE_D$ & Bernoulli bond percolation; expectation conditioned on every edge of $D$ being open. & \S\ref{subsec:constrained}\\
	$\mu_{G,p,q}$ & Extended random-cluster measure, including the connected-conditioned endpoint $q=0$. & \eqref{eq:extended-measure}--\eqref{eq:connected-conditioned}\\
	NC, NA, CNA, $\mathrm{CNA}^+$ & Pairwise negative correlation; negative association; conditional negative association; preservation also under external fields and projections. & \S\ref{subsec:negative-dependence}\\
	\addlinespace
	\multicolumn{3}{@{}l}{\textbf{Partition functions and exceptional configurations}}\\
	$D$, $D_S$ & A forcing set of at most two edges; $D_S=D\cap E(S)$. & \S\ref{subsec:constrained}, \eqref{eq:activity}\\
	$\Zhat_{G,p,q}$, $\Zhat_D$ & Extended partition function; its forced-edge version $\EE_D[q^{k-1}]$. & \eqref{eq:extended-Z}, \eqref{eq:constrained-Z}\\
	$\mathcal K_-$, $k_-$ & Open components of size at most $N/2$, and their number. & \S\ref{subsec:giant}\\
	$\cD$ & Event that no open component has more than $N/2$ vertices. This is an event, unlike the edge set $D$. & \S\ref{subsec:giant}\\
	$T_q$, $R_D$ & Exceptional correction and its conditional expectation. & \eqref{eq:TQ}--\eqref{eq:YR}\\
	$Y_D$ & $\EE_D[q^{k_-}]$, the exact polymer partition function; $\Zhat_D=Y_D+\tau R_D$. & \eqref{eq:YR}--\eqref{eq:Z-Y-R}\\
	$\Lhat_q(e,f)$ & Four-term logarithmic difference of $\Zhat_D$; its sign is the covariance sign. & \eqref{eq:log-correlation}--\eqref{eq:sign-equivalence}\\
	$\Lambda_Y(e,f)$ & The corresponding four-term difference of $\log Y_D$. & \eqref{eq:Lambda-Y}\\
	\addlinespace
	\multicolumn{3}{@{}l}{\textbf{Polymers, modified activities, and clusters}}\\
	$\cP$, $\cP_{\leq}$ & Nonempty induced-connected vertex sets; those of size at most $N/2$. & \S\ref{subsec:polymer-identity}, \eqref{eq:lambda-theta}\\
	$S\sim T$, $\zeta(S,T)$ & Disjoint nonadjacent polymers are compatible; $\zeta=-\one_{\{S\not\sim T\}}$. & \S\ref{subsec:polymer-identity}, \eqref{eq:def-pair-fcn}\\
	$c(S,H)$, $\chi_-(S,H)$ & Number of components of $(S,H)$; indicator that each has size at most $N/2$. & \S\ref{subsec:polymer-identity}\\
	$w_D$, $w$ & Signed activity for forcing set $D$; $w=w_\varnothing$. & \eqref{eq:activity}\\
	$\widetilde w_D$ & $e^{\lambda d m_L(S)}w_D(S)$, only on $\cP_{\leq}$. & \eqref{eq:tilted-activity}\\
	$v_u(S)$, $R_*$ & Size-marked activity $u^{|S|}v(S)$; analytic radius $R_*=e^{\lambda/2}$. & Lemma~\ref{lem:KP-marker}\\
	$b(S)$ & Nonnegative control function in the Koteck\'y--Preiss condition. & \eqref{eq:KP}, Table~\ref{tab:KP-applications}\\
	$I_{\Gamma}$ & incompatibility graph & \S\ref{subsec:ueltschi}\\
	$\Gamma=(S_1,\ldots,S_j)$ & Ordered polymer tuple; repeated labels are allowed. & \eqref{eq:cluster-notation}\\
	$\phit(\Gamma)$ & Ursell coefficient without a factorial. Ueltschi's $\varphi$ equals $\phit/j!$. & \eqref{eq:ursell}\\
	$\wt_D(\Gamma)$ & $\phit(\Gamma)\prod_iw_D(S_i)/j!$. & \eqref{eq:cluster-notation}\\
	$M(\Gamma)$, $\Omega(\Gamma)$ & Total label size $\sum_i|S_i|$ (with multiplicity) and union $\bigcup_iS_i$. & \eqref{eq:cluster-notation}\\
	$N_J(v,b)$ & Number of connected vertex sets containing $v$, of size at most half the graph, with boundary size $b$. & \S\ref{subsec:cut-count}\\
	$\rho(K)$ & Positive block majorant used for large polymers and the no-giant event. & \eqref{eq:rho}\\
	$C_*$, $\beta_d$, $\eta_L$ & Cluster-count constant $32\mathrm e$; $\beta_d=C_*d^2$; $\eta_L=\beta_ds^{d-(L+5)}$. & Lemma~\ref{lem:cluster-count}, \eqref{eq:beta-eta}\\
	\addlinespace
	\multicolumn{3}{@{}l}{\textbf{Shortest paths and graph examples}}\\
	$\ell$, $\nu_\ell(e,f)$ & Edge distance and number of shortest paths oriented from $V(e)$ to $V(f)$. & \S\ref{subsec:path-calculation}\\
	$z$, $\alpha$ & Algebraic parameters for tree activities; no uniform bound on $\alpha$ is used. & \eqref{eq:z-alpha}\\
	$F_m(y)$, $Y_D^P(u)$ & Path subset polynomial and the size-marked path polymer partition function. & \eqref{eq:Fm}, Proposition~\ref{prop:path-leading}\\
	$\kappa_q$ & Positive path factor $p(p+qs)/s^2$. & \eqref{eq:kappa}\\
	$E_q(e,f)$ & Remainder in the logarithmic correlation asymptotic. & \eqref{eq:asymptotic}--\eqref{eq:error-bound}\\
	$\mathfrak M_{n,q,e,f}$ & Positive magnitude of the shortest-path leading term. & \S\ref{subsec:main-proof}\\
	$h_E(G)$ & Normalized edge Cheeger constant; distinct from the unnormalized two-scale constant $h$. & \eqref{eq:normalized-expansion}\\
	$\lambda_2(G)$ & Second adjacency eigenvalue, ordered by value, not absolute value. & Lemma~\ref{lem:spectral-cut}\\
	$\lambda_{\rm abs}(G)$ & Operator norm of adjacency on $\one^\perp$.  & \S\ref{subsec:random-regular}\\
	$B^{\square n}$, $y_B$ & Cartesian power of a fixed regular factor and its product-isoperimetry constant. & Corollary~\ref{cor:products}\\
	$\Omega_{N,d}$, $\mathbb P_{N,d}$ & Simple labeled $d$-regular graphs on $[N]$ and their uniform law.  & \eqref{eq:uniform-regular-law}\\
\end{longtable}
\endgroup

\section*{Statement of AI Use}
The authors used ChatGPT to assist with drafting and revising the manuscript, literature searches, and checking mathematical arguments. The authors reviewed all AI-assisted content and take full responsibility for the manuscript, including its proofs and references.

\bibliography{reg_ref}
\bibliographystyle{plain}

\end{document}